\documentclass[a4paper,10pt,reqno]{amsart}

\usepackage[english]{babel}
\usepackage[latin1]{inputenc}
\usepackage{amsmath,amssymb,amsfonts,amsthm,amscd}
\usepackage[mathscr]{eucal}
\usepackage{latexsym}
\usepackage{multicol}
\usepackage{color}
\usepackage{multirow}
\usepackage{enumitem}

\usepackage{graphicx}

\definecolor{alertmanzano}{rgb}{0.8,0,0.3}

\newcommand{\alertm}[1]{%
	\marginpar{%
		\ifodd\value{page} \raggedright \else \raggedleft \fi
		\footnotesize{\textcolor{alertmanzano}{#1}}
	}
}

\newcommand{\Z}{\mathbb{Z}}

\newcommand{\R}{\mathbb{R}}

\newcommand{\h}{\mathbb{H}}

\newcommand{\E}{\mathbb{E}}

\newcommand{\PSL}{\widetilde{\mathrm{SL}}_2(\mathbb{R})}

\newtheorem{theorem}{Theorem}[section]
\newtheorem*{theorem*}{Theorem}
\newtheorem{proposition}{Proposition}[section]
\newtheorem{corollary}{Corollary}[section]
\newtheorem{lemma}{Lemma}[section]

\theoremstyle{definition}
  \newtheorem{definition}{Definition}[section]

\theoremstyle{remark}
  \newtheorem{remark}{Remark}
\newtheorem{claim}{Claim}

\numberwithin{equation}{section}

\title{On the asymptotic Plateau problem in  $\PSL$}
\date{}

\author{Jes\'us Castro-Infantes}
\address{Universidad de Granada}
\email{jcastroinfantes@ugr.es}


\begin{document}

\begin{abstract}
 We prove some non-existence results for the asymptotic Plateau problem of minimal and area minimizing surfaces in the homogeneous space $\PSL$ with isometry group of dimension 4, in terms of their asymptotic boundary. Also, we show that a properly immersed minimal surface in $\PSL$ contained between two bounded entire minimal graphs separated by vertical distance less than $\sqrt{1+4\tau^2}\pi$ has multigraphical ends. Finally, we construct simply connected minimal surfaces with finite total curvature which are not graphs and  a family of complete embedded minimal surfaces which are non-proper in $\PSL$. 
\end{abstract}

\maketitle
\section{introduction}

The simply connected homogeneous 3-manifolds with isometry group of dimension 4 can be classified as a family of spaces $\mathbb E(\kappa,\tau)$ with $\kappa,\tau\in\R$ ($\kappa-4\tau^2\neq0$). There exists a Riemannian Killing submersion $\Pi: \mathbb E(\kappa,\tau)\to \mathbb M^2(\kappa)$ with bundle curvature $\tau$  over the simply connected complete surface $\mathbb M^2(\kappa)$ of constant curvature $\kappa$, see \cite{Dan,Man}. If $\tau=0$,  one obtains the Riemannian product space $\mathbb M^2(\kappa)\times\R$. If $\tau\neq 0$, one obtains the Berger spheres $(\kappa>0)$, the Heisenberg space $(\kappa=0)$, and $\PSL$, the universal cover of the special linear group  with some special left invariant metrics $(\kappa< 0)$.

In this paper we focus on the space $\PSL$, where we give some new ideas showing that some results about minimal surfaces in the product space $\h^2\times\R$ can be extended to this ambient. All the results in this paper include the case $\tau=0$, which corresponds to $\h^2\times\R$.

%\[\PSL\equiv \h^2\times\R=\{(x,y,t)\in\R^3:\ \lambda(x,y)>0 \},\]
%endowed with metric $g_\lambda=\lambda^2(dx^2+dy^2)+\left( 2\tau(\frac{\lambda_y}{\lambda}dx-\frac{\lambda_x}{\lambda}dy)+dt\right) ^2$.
%
%\begin{enumerate}
%	\item If $\lambda(x,y)=\frac{1}{y}$, we have the half space model, where we identify the asymptotic boundary of $\h^2$ with $\{y=0\}\cup\{\infty\}\equiv \R\cup \{\infty\}$.
%	\item If $\lambda(x,y)=\frac{2}{1-x^2-y^2}$, we have the cylinder model, where we identify the asymptotic boundary of $\h^2$ with $\mathbb S^1$.
%\end{enumerate}
 We identify topologically the space $\PSL$ with $\h^2\times\R$, and we will consider the half space model and the cylinder model for $\PSL$, see section 2. The Riemannian submersion into $\h^2$ reads as $\Pi(x,y,t)=(x,y)$ in these models.
 Considering the product compactification for $\PSL$,  we have that  the asymptotic boundary of $\PSL$ consists of the vertical boundary ${\partial_\infty \h^2 \times\R }$ and the horizontal boundaries $\h^2\times \{ +\infty \}$ and $\h^2\times \{ -\infty \}$. We denote the asymptotic boundary of $\PSL$ as  $\partial_\infty \PSL$.  {\it The asymptotic Plateau} problem consists in: Given $\Gamma$, a finite collection of simple closed curves in the asymptotic boundary of $\PSL$, decide if there is an area minimizing or a minimal surface with asymptotic boundary $\Gamma$. We will call {\it curve}  a finite collection of disjoint simple  closed curves.

   In the last few years the asymptotic Plateau problem has been studied in the product space $\h^2\times\R$, that is, the case  $\tau=0$. Many of these results (discussed below)  are based in the notion of height of a curve, that easily extends to any $\tau\in\R$: 
   \begin{definition}[Definition 1.1 in \cite{KMR}]\label{d:tall}
   	Let $\Gamma$ be a collection of pairwise disjoint simple closed curves in $\partial_\infty \PSL$, and $\Omega=\partial_\infty \PSL\backslash \Gamma$. For each $p\in \partial_\infty \h^2$ let $h_\Gamma(p)$ be the length of the shortest connected component of $(\{p\}\times\R)\cap \Omega$ and define $h_\Gamma=\underset{p\in \partial_\infty\h^2}{\inf}h_\Gamma(p)$. We say that $\Gamma$ is  \emph{tall} if $h_\Gamma(p)>\sqrt{1+4\tau^2}\pi$ for any $p\in\partial_\infty\h^2$.
      \end{definition}
  The number $\sqrt{1+4\tau^2}\pi$ is relevant since if we consider the cylinder model of $\PSL$ and  $\Gamma$ is the disjoint union of two horizontal circles in $\partial_\infty\h^2\times \R$, then there exist rotational catenoids with asymptotic boundary $\Gamma$  if and only if $h_\Gamma<\sqrt{1+4\tau^2}\pi$.
    We emphasize the following contributions when $ \tau=0$:
\begin{itemize}
\item Nelli and Rosenberg  proved in \cite{NR} that for any closed simply curve $\Gamma$ projecting graphically into $\partial_\infty \h^2$, there exists a unique entire minimal vertical  graph in $\h^2\times\R$ with asymptotic boundary $\Gamma$.

\item  Sa Earp and Toubiana   proved in \cite{ST} a general non existence result for minimal surfaces in $\h^2\times \R$, see Theorem \ref{theorem: Asym} for $\tau=0$. They  obtain as a consequence that there are no minimal surfaces with asymptotic boundary a Jordan curve homologous to zero in $\partial_\infty\h^2\times\R$ strictly contained in a slab of width $\pi$. 

\item 	Coskunuzer  proved in \cite{C} that for a tall curve $\Gamma$, there exists an area minimizing surface (possibly disconnected) with asymptotic boundary $\Gamma$. He also gave a non existence result for area minimizing surfaces when the height of the curve is less than $\pi$ in a open arc (see Definition \ref{d:tall}).
	
\item	Kloeckner and Mazzeo considered in \cite{KM} the problem for curves $\Gamma$ with parts in the horizontal asymptotic boundaries $\h^2\times\{\pm \infty \}$, showing that the unique parts of $\Gamma$ in these  horizontal asymptotic boundaries must be geodesics. They also gave an example of a curve $\Gamma$ in the asymptotic boundary whose height is less than $\pi$, for which there is a  minimal surface with asymptotic boundary $\Gamma$ and there is no  area minimizing surface with this asymptotic boundary.

\item 	Ferrer, Martín, Mazzeo and Rodríguez in \cite{FMMR}  proved existence and non-existence results for minimal annuli having two curves in the asymptotic boundary projecting graphically onto $\partial_\infty \h^2$.
\end{itemize}

 In  the case  $\tau\neq 0$, Folha and Peñafiel in \cite{FP} and  Klaser, Menezes and Ramos in a recent work  \cite{KMR} extend some of these results to the space $\PSL$. Folha and Peñafiel proved that for any closed simply curve $\Gamma$ projecting graphically into $\partial_\infty \h^2$, there exists a unique entire minimal vertical  graph in $\PSL$ with asymptotic boundary $\Gamma$.  Klaser, Menezes and Ramos prove that for a \emph{tall} curve $\Gamma$ in $\PSL$, there exists an area minimizing surface (possibly disconnected) with asymptotic boundary $\Gamma$. To this end, they use as barriers the surfaces called \emph{tall rectangles}, whose asymptotic boundary is a rectangle with height $h>\sqrt{1+4\tau^2}\pi$, see Section~3.
  
   They also obtain a non-existence  result for area minimizing surfaces when the height of the curve is less than $(\sqrt{1+4\tau^2}-4\tau)\pi$ in an open arc. This estimate is only valid when $|\tau|<\frac{1}{\sqrt{12}}$, because of the fact that  hyperbolic horizontal translations do not preserve the $t$-coordinate, and they need to send the intersection of a rotational catenoid  with a horizontal slab to a neighbourhood of an ideal point in $\partial_\infty\h^2\times\R$ by a hyperbolic translation controlling the boundary components.

 Here we prove a   general non existence result for minimal surfaces that extends  Theorem 2.1 in \cite{ST} to the case $\tau\neq 0$:
 
 \begin{theorem}\label{theorem: Asym}
 	Let $\Gamma$ be a curve in $\partial_\infty\PSL$ and assume that there exist a vertical line $L$ in $\partial_\infty\h^2\times\R$  and a subarc $\Gamma'\subset\Gamma$  such that: 
 	\begin{enumerate}
 		\item $\Gamma'\cap L \neq \emptyset$ and $\partial \Gamma'\cap L =\emptyset,$
 		\item $\Gamma'$ lies on one side of $L$, and
 		\item $\Gamma'$ is contained in $\{(x,0,t),\ t_0<t<t_0+\sqrt{1+4\tau^2}\pi\}$.
 	\end{enumerate}
 	Then, there is no properly immersed minimal surface in $\PSL$   (with possibly finite boundary)  with asymptotic boundary $\Gamma$.
 	
 \end{theorem}
To this end, we can not use rotational catenoids in the argument as Sa Earp and Toubiana did. The principal problem is that  hyperbolic translations in the cylinder model do not preserve the $t$-coordinate, hence we can not send a compact piece of a rotational catenoid to a neighbourhood of an ideal point in $\partial_\infty\h^2\times\R$  so that the two components of the boundary can be separated by a horizontal open slab. To overcome this issue, we construct a family of compact minimal annuli in  $\PSL$ and work in the half space model using  hyperbolic translations that preserve the $t$-coordinate. Moreover, we can use these annuli to obtain a result as in \cite{KMR}, which gives information for all $\tau\in \R$:
   
   \begin{theorem}\label{minimizing}
   	Let $\Gamma$ be a curve in $\partial_{\infty}\PSL$, and assume that there exists an interval $I\subset\partial_\infty \h^2$ such that $h_\Gamma(p)<\sqrt{1+4\tau^2}\pi$ for all $p\in I$. Then  there are no  area minimizing  surfaces  in $\PSL$  (with possibly finite boundary) with asymptotic boundary $\Gamma$. 
   \end{theorem}
Observe that  Theorem \ref{theorem: Asym} and Theorem \ref{minimizing}  are local so the hypothesis that $\Gamma$ is  disjoint and simple  is necessary only in a neighbourhood.  Moreover, in both theorems the height estimate is sharp except for possibly  the critical case of $h_\Gamma=\sqrt{1+4\tau^2}\pi$.

Another important problem in the theory of minimal surfaces is to classify such surfaces by their topological type. Collin, Hauswirth and Rosenberg proved in \cite{CHR} that a properly immersed minimal surface in  $\h^2\times\R$ of finite topology  inside a slab of width  strictly less than $\pi$ has multigraphical ends. Moreover, if the surface is embedded, it has graphical ends; and if in addition it is simply connected, then it is an entire graph. This result is known as The Slab Theorem. Later, Lima in \cite{Lima} extended this result to the case  $\tau\neq 0$ by replacing the slab region by a \emph{ generalized slab region}, according to the following definition:

\begin{definition}[Definition 1 \cite{Lima}]\label{d:slab}
We say that a region $\mathcal R$ in $\PSL$ is a generalized slab if the following conditions are satisfied:
\begin{enumerate}

	\item $\mathcal R$ is a domain bounded by two disjoint entire vertical graphs $S_1$ and $S_2$ with bounded height, and the tangent planes of $S_1$ and $S_2$ are bounded away from the vertical, this is, there exists $c>0$ such that $\nu_i^2>c$, being $\nu_i=\langle N_i,\partial_t\rangle$ the angle function and $N_i$ the unit normal vector to the surface $S_i$, $i=1,2$.
	
	\item There is a $\mathcal C^1$ map $\Psi :\mathcal R\times (\mathbb S^1\times [-1,1]) \mapsto \PSL$ such that for each $p\in \mathcal R$ we have that $C(p)= \Psi (p,\mathbb S^1\times [-1,1])$ is a minimal annulus containing $p$ and its two boundary curves lie one above $\mathcal R$ and the other  below $\mathcal R$. Moreover, any two annuli of the family $\{C(p):\ p\in \mathcal R  \}$ are isometric
	to each other.

\end{enumerate}

\end{definition}

We prove a Slab Theorem in $\PSL$ when $\mathcal R$ is the region between an entire minimal graph $G_1$ in the cylinder model whose asymptotic boundary is a closed graphical curve over $\partial_\infty\h^2$ and it is bounded away from the vertical, and its translated copy $G_2:=G_1+(0,0, \sqrt{1+4\tau^2}\pi-\epsilon)$, where $\epsilon$ is any positive number less than $\sqrt{1+4\tau^2}\pi$.

\begin{theorem}\label{t:slab}
	Let $M$ be a properly immersed minimal surface of finite topology in $\PSL$  contained in the region between $G_1$ and $G_2$. Then, each  end of M is a multigraph. Moreover:
	
	\begin{enumerate}
		\item If $M$ is embedded, then the neighbourhood of any  of its ends  is a graph over the complement of a disk in $\h^2$.
		
		\item If $M$ is embedded and has only one end, then $M$ is an entire graph.
	
	\end{enumerate}
	
\end{theorem}

To this end, we construct a continuous family of minimal annuli that plays the role of the  the map $\Psi$ in Definition \ref{d:slab}, and we conclude using  the same  ideas as in the proof of the Slab Theorem in \cite{Lima}.

 Related to this topic are the constructions of minimal surfaces in $\h^2\times\R$ in \cite{RP}.  They solved the asymptotic Plateau problem for some special curves $\Gamma$ composed of vertical straight lines in $\partial_\infty\h^2\times\R$ and horizontal geodesics in $\h^2\times\{\pm\infty\}$. These examples are interesting since they are non-flat complete simply connected minimal surfaces with finite total curvature, and they are not  graphs. In Section 7.1 we construct the analogous examples in $\PSL$ which have also finite total curvature due to Theorem  8 in \cite{HMR}. Recently, Collin, Hauswirth and Nguyen in \cite{CHN} have constructed minimal annuli with finite total curvature via variational methods. It would be interesting  to obtain examples with higher topology similar to the minimal $k$-noids in \cite{MorRod}, or with positive genus as in \cite{CasMan,MMR}.
 
 Similar techniques as in \cite{RP}  are used in \cite{RT} to construct complete embedded minimal surfaces in $\h^2\times \R$ which are non proper. These surfaces  are interesting in relation to the Calabi-Yau conjecture for embedded minimal surfaces. This conjecture says that any complete embedded minimal surface in $\R^3$ is necessarily proper.  Colding and Minicozzi in \cite{CM} showed that any complete minimal surface embedded in $\R^3$ with finite topology is proper. Meeks, Pérez and Ros proved in \cite{MPR} that any complete  minimal surface embedded in $\R^3$ with  an infinite number of ends and  finite genus is proper if and only if it has at most two simple limit ends if and only if it has a countable number of limits ends. The examples constructed by Rodríguez and Tinaglia in \cite{RT} show that the conjecture does not hold in $\h^2\times\R$. We construct the analogous examples in Section 7.2 showing that the conjecture does not hold in $\PSL$ either, as it was expected.

 The paper is organized as follows: In Section 2 we recall the isometries  of $\PSL$.  In Section 3 we describe some invariant minimal surfaces obtained as vertical graphs in the half space model of $\PSL$. In Section 4 we construct a family of compact minimal annuli which will be used to prove Theorem \ref{theorem: Asym} and Theorem  \ref{minimizing} in Section~5. Section 6 is devoted to prove Theorem \ref{t:slab}. Finally, in Section 7 we extend the constructions of minimal surfaces of \cite{RP} and \cite{RT} to the space $\PSL$.

{\bf Acknowledgements.} The author would like to thank Magdalena Rodríguez, José Miguel Manzano and Ana Menezes  for very useful discussions and suggestions on this topic and to Laurent Mazet for his help in Section 4.  This research is supported by Spanish MINECO--FEDER research project MTM2017-89677-P and by a FPU grant from the Spanish Ministry of Science and Innovation.

\section{The space $\PSL$}
Given $\kappa,\tau\in\R$, the space $\E(\kappa,\tau)$ is the unique
	simply connected oriented 3-manifold admitting a Riemannian submersion
	with constant bundle curvature $\tau$ over $\mathbb{M}^2(\kappa)$, the
	simply connected surface with constant curvature $\kappa$, whose
	fibers are the integral curves of a unit Killing vector field
	$\xi$. The bundle curvature can be characterized by the equation
	$\overline\nabla_X\xi=\tau X\times\xi$, where $\overline\nabla$ stands
	for the Levi-Civita connection and its sign depends on the
	orientation of $\E(\kappa,\tau)$ by means of the cross product
	$\times$, see \cite{Man} for details.

In this section we  describe the isometries of the space $\PSL$. We will assume that $\kappa=-1$ after rescaling the metric and define  
\[\PSL =\{(x,y,t)\in\R^3:\ \lambda(x,y)>0 \},\]
endowed with metric $g_\lambda=\lambda^2(dx^2+dy^2)+\left( 2\tau(\frac{\lambda_y}{\lambda}dx-\frac{\lambda_x}{\lambda}dy)+dt\right) ^2$, where

\begin{enumerate}
	\item If $\lambda(x,y)=\frac{1}{y}$, we have the half space model  $\mathcal H$. In this model we identify the asymptotic vertical  boundary  with $\left( \{y=0\}\cup\{\infty\}\right) \times\R\equiv \left( \R\cup \{\infty\}\right) \times\R$. 
	\item If $\lambda(x,y)=\frac{2}{1-x^2-y^2}$, we have the cylinder model $\mathcal C$, where we identify the asymptotic vertical boundary  with $\mathbb S^1\times\R$.
\end{enumerate}
In these models the Riemannian submersion reads as $\Pi(x,y,t)=(x,y)$ and the Killing vector field is $\xi=\partial_t$.
In the sequel, we  choose the complex coordinate $z=x+iy$.
\begin{proposition}\label{prop:iso}
	The maps $\phi:\mathcal H\to \mathcal C$ and  $\psi:\mathcal C\to \mathcal H$   given by:
	\begin{equation}\label{iso1}
	\phi(z,t)=\left( \frac{z-i}{z+i},t-4\tau\arctan\left( \frac{x}{y+1}\right) \right)
	\end{equation}
 and 
 \begin{equation}\label{iso2}
	\psi(z,t)=\left( \frac{i+ iz}{1-z},t -4\tau\arctan\left( \frac{y}{1-x}\right) \right), 
 \end{equation}
are isometries between the half space model and the cylinder model of $\PSL$.
\end{proposition}

Note that these isometries extend to the asymptotic boundary. We show in Figure~\ref{Circle} how horizontal lines in the asymptotic boundary of $\PSL$ in the half space model are transformed by the isometry $\phi$, and also how circles in the asymptotic boundary of $\PSL$ in the cylinder model transform  by the isometry $\psi$.

\begin{figure}[htb]

\begin{center}
\includegraphics[height=6cm]{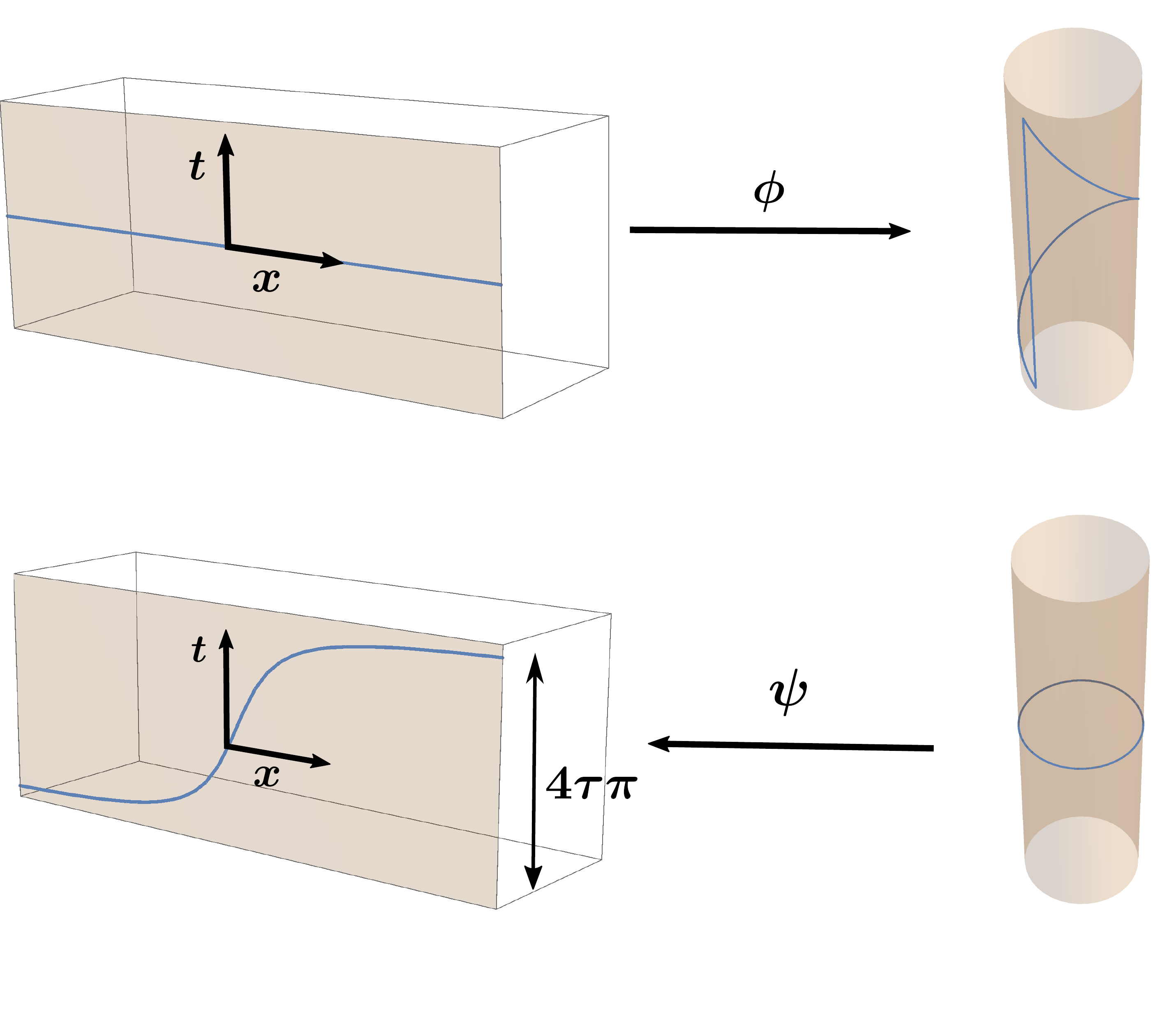}
\end{center}
	\caption{On the top, the image by $\phi$ of a horizontal straight line in the asymptotic boundary of the half space model of $\PSL$. On the bottom,  the image by $\psi$ of a circle in the asymptotic boundary of the cylinder model of $\PSL$.  }\label{Circle}
\end{figure}

\
\

Let $f:\h^2\to\h^2$, $f(z)=\frac{az +b}{cz +d}$ with $ad-bc=1$, $a,b,c,d\in \R$, be a positive isometry of $\h^2$ in the half space model. Then, $f$ can be lifted to an isometry  $\bar{f}$ of $\PSL$: 

\begin{equation}\label{isometrias}
\bar f(z,t)=\left( f(z),t-2\tau \text{arg}f'(z) \right) 
\end{equation}
satisfying $f\circ \Pi=\Pi \circ \bar f$, where $\Pi$ is the Riemannian submersion over $\h^2$. It is unique up to vertical translations. See for instance Proposition 2.1 in  \cite{Penafiel} and \cite{Younes}. Here, we give explicit expressions for the isometries in terms of the parameters $a, b, c, d\in \R$.

\begin{enumerate}
	\item If $c=0$ $(a d= 1)$, we have the isometries which fix $\infty$, and then the lifted isometry in $\PSL$ is given by the map: 
	\[\bar{f}(z,t)=\left( \frac{az+b}{d},t+t_0\right),  \]
	 for some $t_0\in \R$. Up to a vertical translation we can choose $t_0=0$. This isometry extends to the asymptotic boundary  as:
		\[\bar{f}(x,0,t)=\left( \frac{ax+b}{d},0,t \right).  \]

	\item If $c\neq 0$, we have the isometries  sending the ideal point $z=-\frac{d}{c}$ to $\infty$, given by:
	\begin{equation}\label{isometria}
	\bar{f}(x,y,t)=\left( \frac{(d+c x)(-1+a(d+cx))+ac^2y^2}{c(d+cx)^2+c^3y^2},\frac{y}{(d+cx)^2+c^2y^2},t-4\tau \arctan\left(\frac{cx+d}{c y} \right)+t_0 \right) 	 
\end{equation}

	 Choosing $t_0=0$, they extend to the asymptotic boundary  as 
	 \[ \bar{f}(x,0,t)=	  \left\lbrace
	 \begin{array}{cc}
	(\frac{ax+b}{cx+d},0,t+2\tau \pi)  &  x<\frac{-d}{c}, \\
	 
	 (\frac{ax+b}{cx+d},0,t-2\tau \pi) &  x>\frac{-d}{c}.
	 \end{array}                
	 \right. 	 \]
\end{enumerate}
Note that the extended  $\bar{f}$ is not a continuous function as a function from $\R^2$ to $\R^2$, since it has a jump discontinuity  at the ideal point $x_0=\frac{a}{c}.$  Figure~\ref{plano_isometria} shows the image of the minimal slice $\mathcal S$ of equation $\{t=0\}$ by an isometry which sends the ideal point $(0,0)$ to $\infty$. It can be parametrized as the entire vertical graph of the function $(x,y) \mapsto 4\tau \arctan(\frac{x}{y})$.  The asymptotic boundary here consists of two horizontal half straight lines along with a vertical segment of length $4\tau\pi$. 
	\begin{figure}[htb]
	\begin{center}
		\includegraphics[height=4cm]{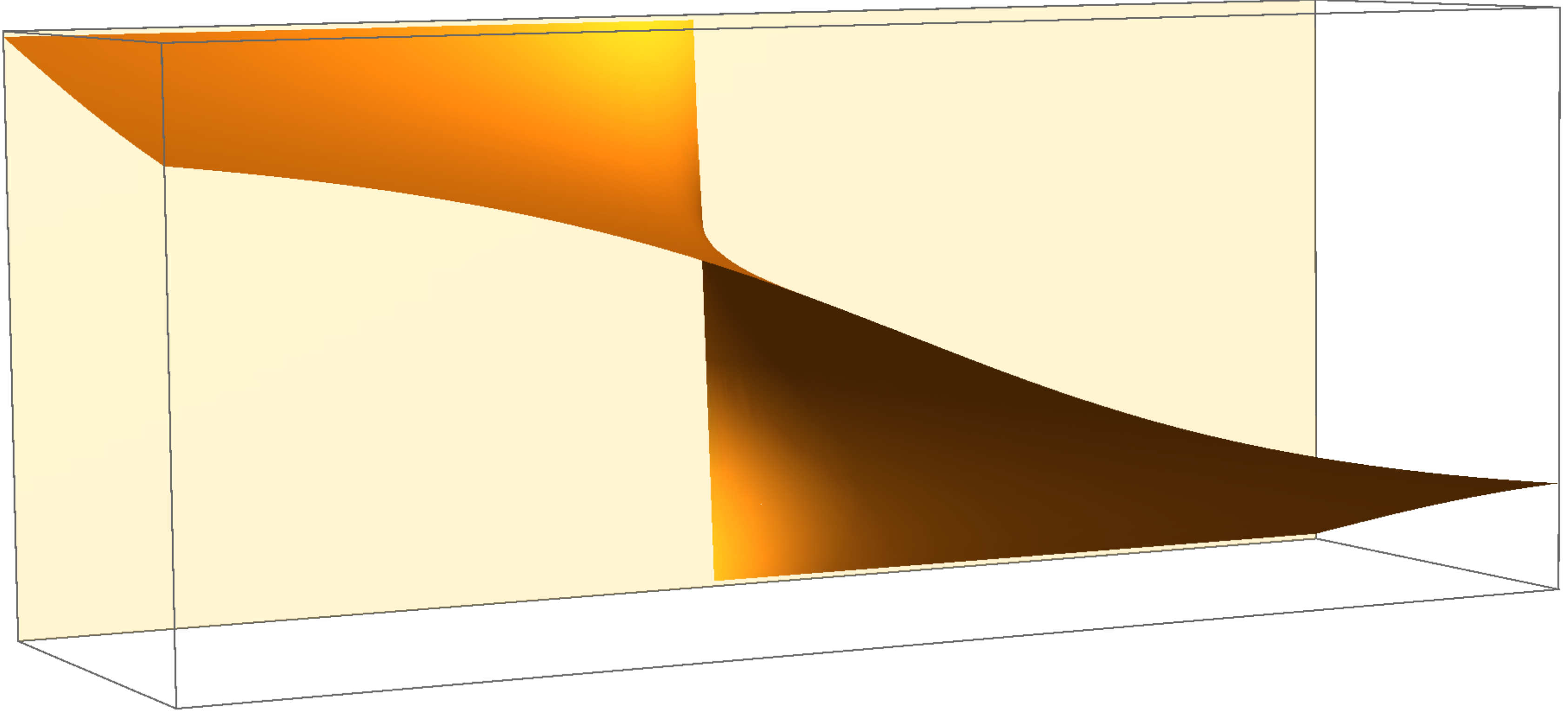} 
		
		\caption{The image of a horizontal plane by the isometry with parameters $a=0,\ b=-1,\ c=1$, $d=1$ and $\tau=1/2$,  in the half space model of $\PSL$.} \label{plano_isometria}
	\end{center}
\end{figure}

 We also  emphasize some  isometries of $\PSL$ that can be expressed in a simple way:
\begin{enumerate}
	\item Vertical translations: $(x,y,t)\mapsto (x,y,t+t_0).$
	
	\item Hyperbolic translations: In the half space model  hyperbolic translations along the horizontal geodesic $\{x=c\}$ is the map  given by the map \newline ${(x,y,t)\mapsto(c+\lambda (x-c),\lambda y,t)}$,  $\lambda>0$.
	
	\item Parabolic translations: In the half space model  parabolic translations along the horocycles $\{y=c^2\}$ are given by the map $(x,y,t)\mapsto ( x+a,y ,t)$,  $a\in\R$.

	\item Rotations: In the cylinder model  rotations with respect to the $t$-axis are Euclidean rotations with  respect to the $t$-axis.

\end{enumerate}

There are not horizontal surfaces in $\PSL$, that is, surfaces everywhere orthogonal to the unit Killing vector field.
The surface $\{t=0\}$ is the so-called minimal umbrella centered in $(0,0,0)$ in the cylinder model $\mathcal C$, i.e the union of all horizontal  geodesics through the point $(0,0,0)\in \mathcal C$. However, the minimal slice $\mathcal S$ of equation $\{t=0\}$ in the half space model is not an umbrella if $\tau\neq 0$, but it can be understood as an umbrella with center an ideal point.

\begin{proposition}
	The  minimal slice $\mathcal S$   is the limit surface of umbrellas under the action of a $1$-parameter group of hyperbolic  translations.
\end{proposition}

\begin{proof}
 Consider the umbrella $\{t=0\}$ in the cylinder model parametrized as:
	\[X(x,y)=\left( \frac{-1+x^2+y^2}{x^2+(1+y)^2},\frac{-2x}{x^2+(1+y)^2},0\right) ,\  x\in\R\  \text{and}\  y>0.
	\]
	 The image by the isometry \eqref{iso2} is the surface $\left( \psi\circ X\right) (x,y)=(x,y,4\tau\arctan(\frac{x}{y+1}) )$. The hyperbolic translation $(x,y,t)\mapsto(\lambda x,\lambda y,t)$ applied to $\psi\circ X$ gives the surface reparametrized by $Y_\lambda(x,y)=( x,y,4\tau\arctan( \frac{x}{y+\lambda}) ) $, $x\in \R$, $y>0$. Then the limits of $Y_\lambda$ when $\lambda\to 0$ and $\lambda\to +\infty$ are respectively: $Y_0(x,y)=(x,y,4\tau\arctan(\frac{x}{y}))$ and $Y_{\infty}(x,y)=(x,y,0)$.
	$Y_0$ is the image by the isometry which sends $(0,0)$ to $\infty$ of the  minimal slice $\mathcal S$, see Figure~\ref{plano_isometria}, and $Y_\infty$ is the minimal slice $\mathcal S$.	
\end{proof}

\section{ Invariant minimal surfaces}
In this section we will describe some invariant minimal surfaces giving explicit expressions of them  as  vertical graphs or  union of graphs of functions ${u:\Omega\subset\h^2\to \R}$ in the half space model. These surfaces have been also described in \cite{FP} and \cite{Penafiel}. 

A (vertical) graph in $\E(-1,\tau)$ is a section of the submersion
$\Pi:\E(-1,\tau)\to\mathbb \h^2$ defined over some domain
$\Omega\subset\h^2$. The mean curvature of the graph of a
function $u\in C^2(\Omega)$ over a zero section
$F_0:\Omega\to\E(-1,\tau)$ can be expressed in divergence form as
\begin{equation}
2H=\mathrm{div}\frac{Gu}{\sqrt{1+\|Gu\|^2}},
\end{equation}
 where the divergence
and norm are computed with respect to the metric of
$\h^2$ and $Gu$ is a vector field in
$\h^2$ called the generalized gradient.
We take as zero section $F_0(x,y)=(x,y,0)$, so the graph is
parametrized in terms of $u$ as
\begin{equation}\label{eqn:graph-parametrization}
F_u(x,y)=(x,y,u(x,y)),\qquad (x,y)\in\Omega,
\end{equation}
and hence $Gu=(u_x+\tau y\lambda^{-1})\partial_x+(u_y-\tau
x\lambda^{-1})\partial_y$.

Therefore, the minimal surface equation for a graph $F_u$ in the half space model is: 
\begin{equation}\label{minimaleq}
y u_y^3-(1+4\tau^2+yu_x(-4\tau +yu_x))u_{yy}+u_y(-2\tau+ yu_x)(u_x+2y u_{xy})-u_{xx}-y^2u_y^2u_{xx}=0.
\end{equation}
We  can solve this equation in some particular cases, obtaining symmetric minimal surfaces. In all the cases we choose the additive constant of the vertical translation as $0$. 
\begin{enumerate}
	\item  If $u(x,y)=u(y)$, then we can solve  Equation \eqref{minimaleq}  obtaining a 1-parameter family of complete minimal  bigraphs, given by the union of the two solutions $u_d^\pm(x,y)=\pm \sqrt{1+4\tau^2}\arcsin(d y)$, as shown in Figure \ref{Benoit}. The asymptotic boundary of these surfaces consists of two horizontal lines at distance $\sqrt{1+4\tau^2}\pi$ from each other. The parameter $d$ reflects a hyperbolic translation of the surface. The limit surface when $d\to +\infty$ consists of two slices, and the limit when $d\to0$ is the open subset of the ideal boundary given by $\{(x,0,t),\ x\in \R,\ 0<t<\sqrt{1+4\tau^2}\pi\}$.
\begin{figure}[htb]
\includegraphics[height=5cm]{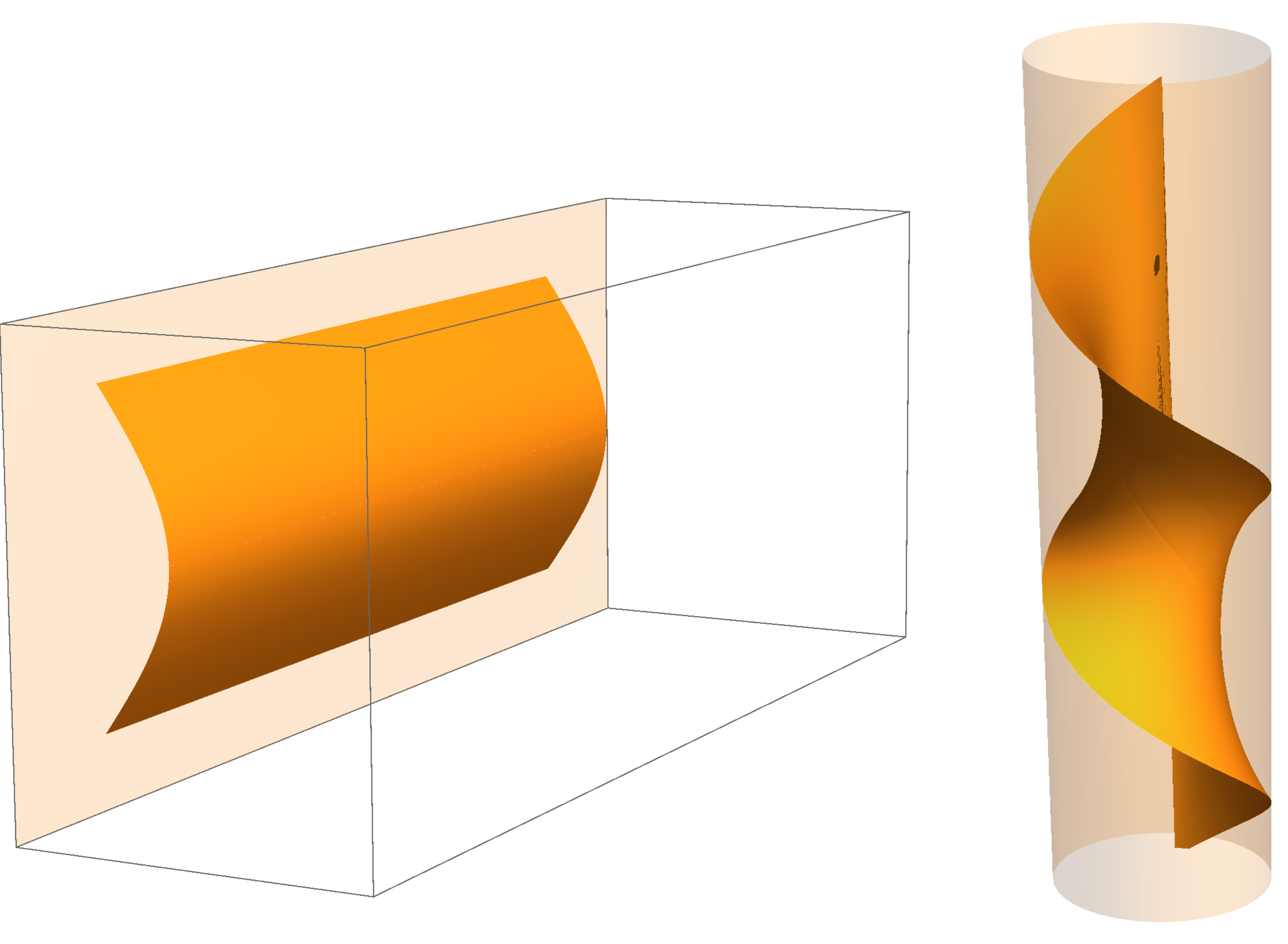} 

\caption{The surface $ u_{d }^\pm$ with $d=1$ and $\tau=1/2$ in the half space model (left) and the cylinder model (right).}\label{Benoit}
\end{figure}

\item If  $u(x,y)=u(y)+l x$ for some $l\in\R$,   Equation \eqref{minimaleq} reduces to:
\begin{equation}
yu'(y)^3-(1+(ly-2\tau)^2)u''(y)+ l(l y-2\tau)u'(y)=0.
\end{equation}

We can solve this equation obtaining the tilted planes $u(x,y)=lx$, as well as  a 2-parameter family of complete minimal surfaces $v_{dl}^\pm(x,y)=v_d^\pm(y)+lx$ where
 \[v_d^\pm(y)=\pm \displaystyle\int_0^y\dfrac{\sqrt{1+(lt-2\tau)^2}}{\sqrt{d^2-t^2}}dt,\  0<y<d.\]
   These surfaces are obtained as  the union of the two solutions corresponding to the choice of the sign $\pm$, see Figure~\ref{TiltedBenoit}. Up to vertical translations, their asymptotic boundary  consists of the two lines of equation $t=lx$ and $t=lx+2v_d^+(d)$ lying in the plane $y=0$. We have the following estimate for $v_d^+(d)$:
\[ v_d^+(d)=\displaystyle\int_0^d\dfrac{\sqrt{1+(lt-2\tau)^2}}{\sqrt{d^2-t^2}}dt\geq \displaystyle\int_0^d\dfrac{1}{\sqrt{d^2-t^2}}dt=\frac{\pi}{2}. \]
 We have  additional information in the following cases:
 
 \begin{enumerate}
 	
 	\item If $l\tau\leq0$, then: 

\begin{align*}
 	v_d^+(d) & =\displaystyle\int_0^d\dfrac{\sqrt{1+(lt-2\tau)^2}}{\sqrt{d^2-t^2}}dt\geq\displaystyle\int_0^d\dfrac{\sqrt{1+(2\tau)^2}}{\sqrt{d^2-t^2}}dt=\frac{\sqrt{1+4\tau^2}\pi}{2},
 	\\
 		v_d^+(d)& =\displaystyle\int_0^d\dfrac{\sqrt{1+(lt-2\tau)^2}}{\sqrt{d^2-t^2}}dt\leq \displaystyle\int_0^d\dfrac{\sqrt{1+(ld-2\tau)^2}}{\sqrt{d^2-t^2}}dt=\frac{\sqrt{1+(ld-2\tau)^2}\pi}{2}.
\end{align*}

 \item  If $l\tau>0$, then:
 \begin{itemize}
  \item If $0<l\tau<\frac{4\tau^2}{d}$ then as $0<t<d$ we have that $l \tau<\frac{4\tau^2}{d}<\frac{4\tau^2}{t}$, whence   $l^2t^2-4lt\tau<0$ and consequently $(l t -2\tau)^2< +4\tau^2$. This gives: 
  \[v_d^+(d)=\displaystyle\int_0^d\dfrac{\sqrt{1+(lt-2\tau)^2}}{\sqrt{d^2-t^2}}dt\leq\displaystyle\int_0^d\dfrac{\sqrt{1+4\tau^2}}{\sqrt{d^2-t^2}}dt=\frac{\sqrt{1+4\tau^2}\pi}{2}.\]
  \item If  $|l d|<2|\tau|$, then $(lt-2\tau)^2>(ld-2\tau)^2$, and consequently:
  	 	\[v_d^+(d)=\displaystyle\int_0^d\dfrac{\sqrt{1+(lt-2\tau)^2}}{\sqrt{d^2-t^2}}dt\geq \displaystyle\int_0^d\dfrac{\sqrt{1+(ld-2\tau)^2}}{\sqrt{d^2-t^2}}dt=\frac{\sqrt{1+(ld-2\tau)^2}\pi}{2}.
  	 \]
  	
 \end{itemize} 	 

 \end{enumerate} 
Fix $l\in \R$. The limit surface when $d\to 0$ is the region of the asymptotic boundary   $\{ (x,0,t):  x<t<lx+ d_0 \}$ with $d_0=\sqrt{1+4\tau^2}\pi $. Note that we have  opposite inequalities depending on the sign of $\tau$, this can be understood  by considering the isometry between the half space models of $\mathbb E(-1,\tau)$ and $\mathbb E(-1,-\tau)$ given by the map $(x,y,t)\mapsto (x,y,-t)$. 
\begin{remark}
	Item (b) shows that there exist examples of minimal surfaces such that the  height of the asymptotic curve is less than $\sqrt{1+4\tau^2}\pi$ at every point. Theorem \ref{minimizing} shows that when the height is less than $\sqrt{1+4\tau^2}\pi$ these examples are not area minimizing. Also these surfaces are contained between two entire minimal   graphs with unbounded height in the half space model separated by a distance less than $\sqrt{1+4\tau^2}\pi$. This shows that the hypothesis of being a bounded graph in  Theorem \ref{t:slab} is necessary.  
	
\end{remark}
\begin{figure}[htb]
	\includegraphics[height=5cm]{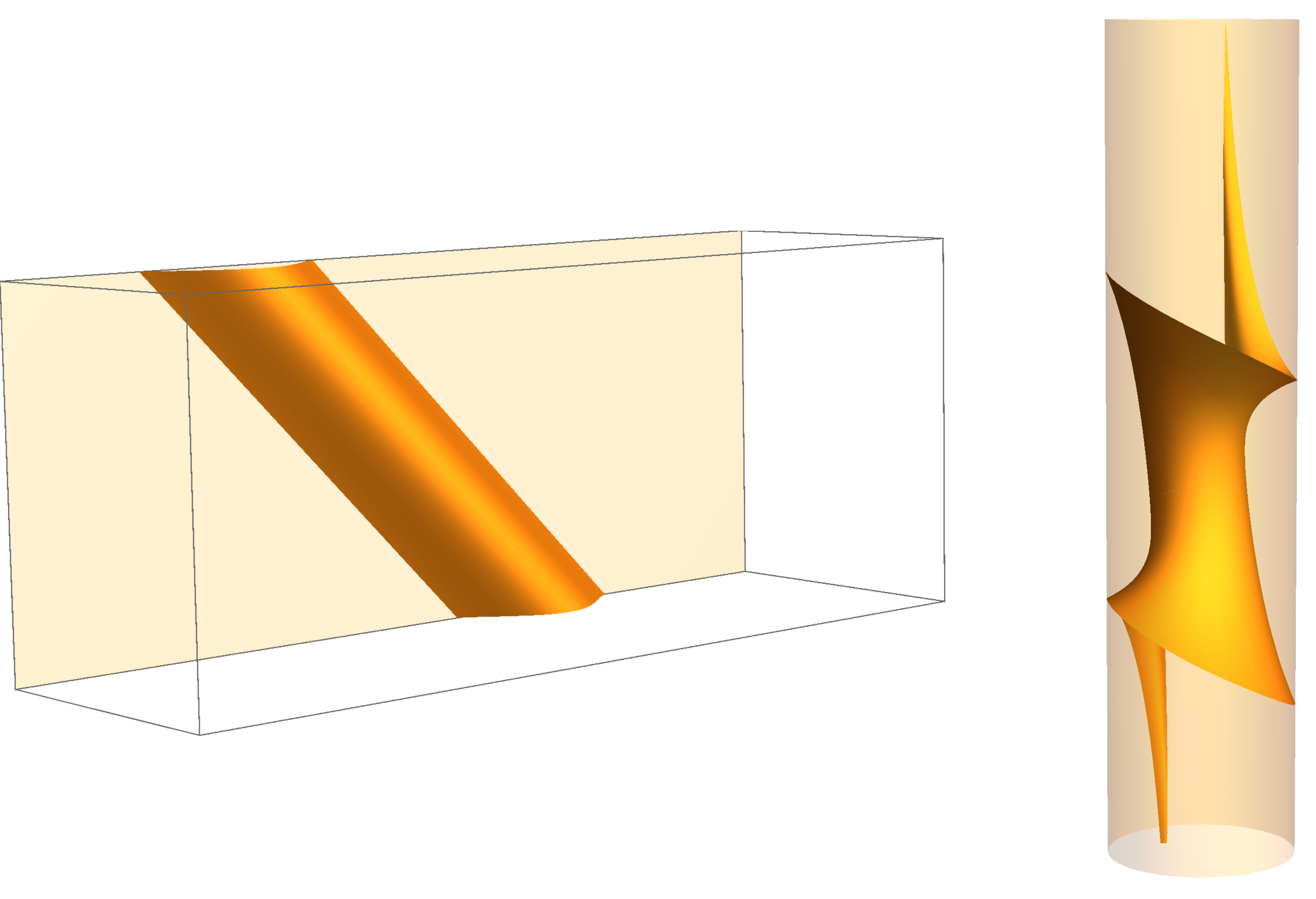} 
	\caption{The surface $ v_{dl }^\pm$ with $d=1$, $l=1$ and $\tau=1/2$ in the half space model (left) and in the cylinder model (right).}\label{TiltedBenoit}
\end{figure}

\item If $u(x,y)=u( \frac{x}{y})$ and we write $s=\frac{x}{y}$, then  Equation \eqref{minimaleq} reduces to:
\begin{equation}\label{eq:tall}
2s(1+4\tau^2)u'(s)-6s\tau u'(s)^2+(s^3+s)u'(s)^3+(1+s^2(1+4\tau^2))u''(s)=0.        
\end{equation}

The solutions of  Equation \eqref{eq:tall} are given by: 
\[u_c^\pm(s)= 2\tau \arctan(s)\pm \displaystyle\int_{c_0}^s     \frac{\sqrt{1+(1+4\tau^2)t^2}}{(t^2+1)\sqrt{c(t^2+1)-1}}dt,\  s\in I(c)                                   ,\] 
 
 where $I(c)\subset \R$ is an interval that depends on $c$. We have three subcases:
\begin{itemize}[]
\item If $c>1$, then $I(c)=\R$ and we have entire graphs whose asymptotic boundary consists of two straight lines joined by a vertical segment contained in $\{(0,0)\}\times\R$. Observe that, when $c=\frac{1+4\tau^2}{4\tau^2}$ we obtain, up to vertical translations, the solutions $u_c^-(s)=0$ and $u_c^+(s)=4\tau\arctan(s)$, i.e., a minimal slice $\mathcal S$ and its image by the isometry which sends the point $(0,0)$ to $\infty$ respectively, see Figure~\ref{plano_isometria}. We have also the limit case when $c\to\infty$, $u_\infty(s)=2\tau \arctan(s)$, which is the surface composed by the geodesic $\{(0,y,0):\ y>0\}$ and all the geodesics orthogonal  to this one. 
\item If $c=1$, then $I(c)=\R^+$ and we have the two solutions:
	\[u_1^\pm(s)= 2\tau \arctan(s)\pm \displaystyle\int_{1}^s     \frac{\sqrt{1+(1+4\tau^2)t^2}}{(t^2+1)t}dt\ s>0 .\]
	The asymptotic boundary of the surface $u_1^+$ (up to vertical translation) consists of the line $\{(x,0,0):x>0\}$  and the vertical line ${\{(0,0,t): t<0 \}}$ in the vertical asymptotic boundary $\partial_\infty\h^2\times\R$, and the horizontal geodesic $\{(0,y,-\infty):\ y>0\}$ in the horizontal asymptotic boundary $\h^2\times\{-\infty\}$, see Figure~\ref{casoc1}. The asymptotic boundary of the surface $u_1^-$ is analogous, with the horizontal geodesic $\{(0,y,+\infty): y>0\}$ in the horizontal asymptotic boundary $\h^2\times\{+\infty\}$.
\begin{figure}[htb]
	\includegraphics[height=5cm]{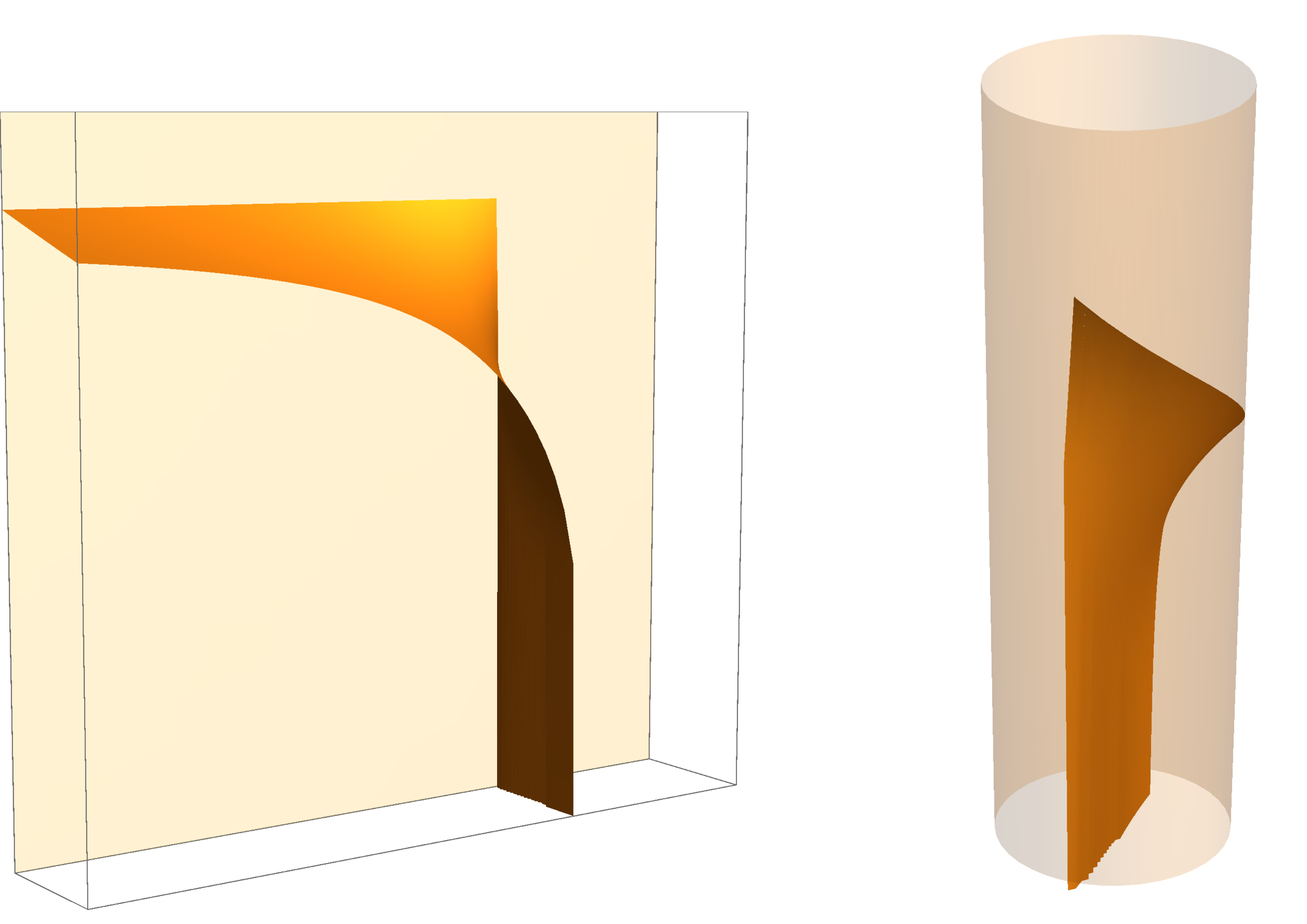} 

	\caption{The surface $u_1^+$ for $\tau=\frac{1}{2}$ and $c=1$ in the half space model (left) and  in the cylinder model   (right).}\label{casoc1}
\end{figure}

\item If $0<c<1$, then $I(c)=({\sqrt{\frac{1-c}{c}}},+\infty)$ and the solution is given by:

	\[u_c^\pm(s)= 2\tau \arctan(s)\pm \displaystyle\int_{\sqrt{\frac{1-c}{c}}}^s  \frac{\sqrt{1+(1+4\tau^2)t^2}}{(t^2+1)\sqrt{c(t^2+1)-1}}dt,\]
	 where  $s> {\sqrt{\frac{1-c}{c}}}.$
	We have that $(u_c^\pm)'(\sqrt{\frac{1-c}{c}})=\pm\infty$, and we can complete the surface as the union of the two solutions, see Figure~\ref{TallRec}. The asymptotic boundary, up to vertical translations, are the lines $\{(x,0, u_c^\pm(+\infty)):\ x>0 \}$ and the vertical segment $\{ (0,0,t):\  u_c^-(+\infty)<t< u_c^+(+\infty)   \} $. Observe that, the asymptotic boundary of these surfaces can be sent by an appropriate isometry to a rectangle. This surfaces are known in the literature as tall rectangles, since their asymptotic boundaries are  rectangles of height $h>\sqrt{1+4\tau^2}\pi$, see \cite{ FP,Lima}.

 \begin{figure}[htb]
	\includegraphics[height=4cm]{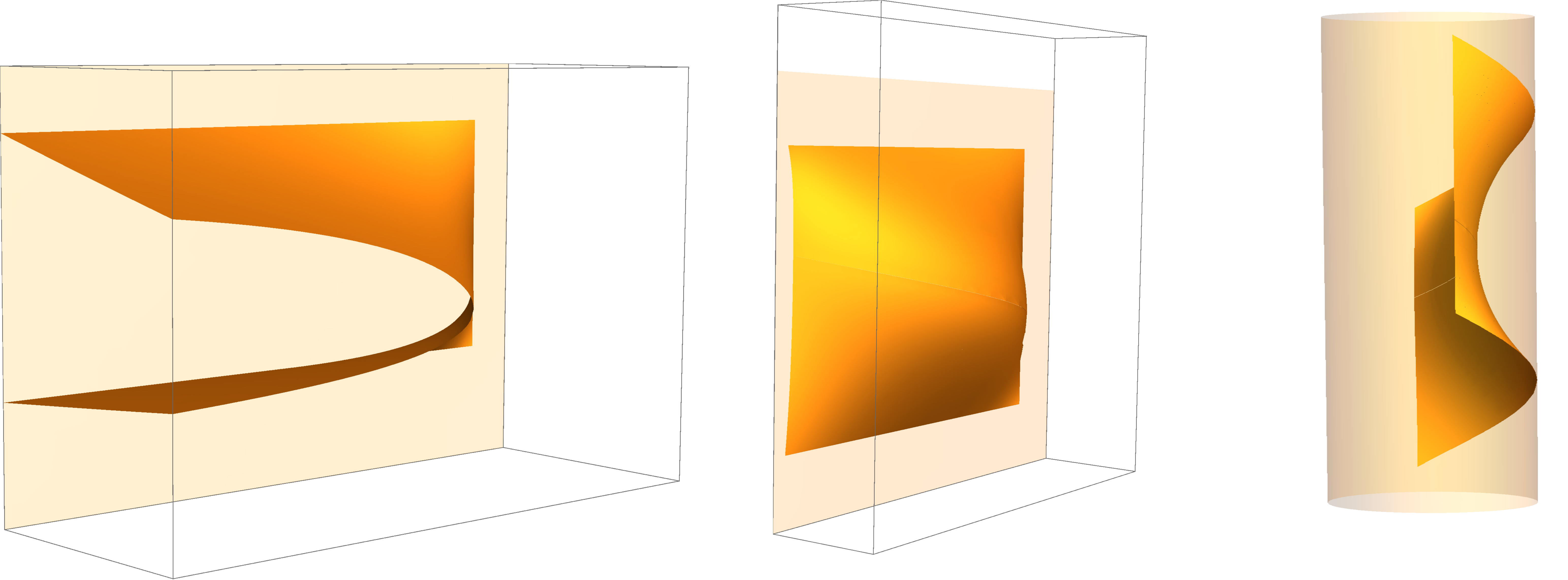} 
	
	\caption{The tall rectangle for $\tau=\frac{1}{2}$ and $c=\frac{1}{2}$ and an isometric copy with asymptotic boundary a rectangle in the half space model (left and center) and the tall rectangle   in the cylinder model   (right).}\label{TallRec}
\end{figure}

\end{itemize}

\item If $u(x,y)=h\left( r\right)+ 4 \tau \arctan\left(\frac{x}{y+1} \right)$, where $r=\tfrac{(y-1)^2+x^2}{x^2+(y+1)^2}$, then  Equation  \eqref{minimaleq} reduces to:
\begin{equation}\label{eq:catenoid}
(1+2\tau^2r)h'(r)+\tfrac{1}{2}(r-3)(r-1)rh'(r)^3+(1+4\tau^2r)rh''(r)=0.        
\end{equation}
 The trivial  solution $h\equiv 0$  corresponds to $u(x,y)=4\tau \arctan\left(\frac{x}{y+1} \right)  $, which is the umbrella centered at $(0,1,0)$. The rest of the solutions are  given by 
\[h_c^\pm(r)=\pm \displaystyle\int_{r_0(c)}^{r}\frac{\sqrt{1+4t\tau^2}}{\sqrt{t(-3+t^2+c t-4t \log(t))}}dt,\  r\in(r_0(c),1),\]
where $r_0(c)$ is the unique real solution of  $-3+t^2+c t-4t \log(t)=0$ in $(0,1)$.
Note that $(h_c^\pm)'  (r_0(c))=\pm\infty$, so we can complete the surface as the union of the two solutions.
  These surfaces are known in the literature as catenoids, see \cite{Penafiel}. The asymptotic boundary of these surfaces consists of the curves $\{x,0,4\tau\arctan(x)+h^+(1) \}$ and $\{x,0,4\tau\arctan(x)-h^+(1) \}$, where $h^+(1)<\sqrt{1+4\tau^2}\frac{\pi}{2}$.
 
 \begin{figure}[htb]
 	\begin{center}
 	\includegraphics[height=4.5cm]{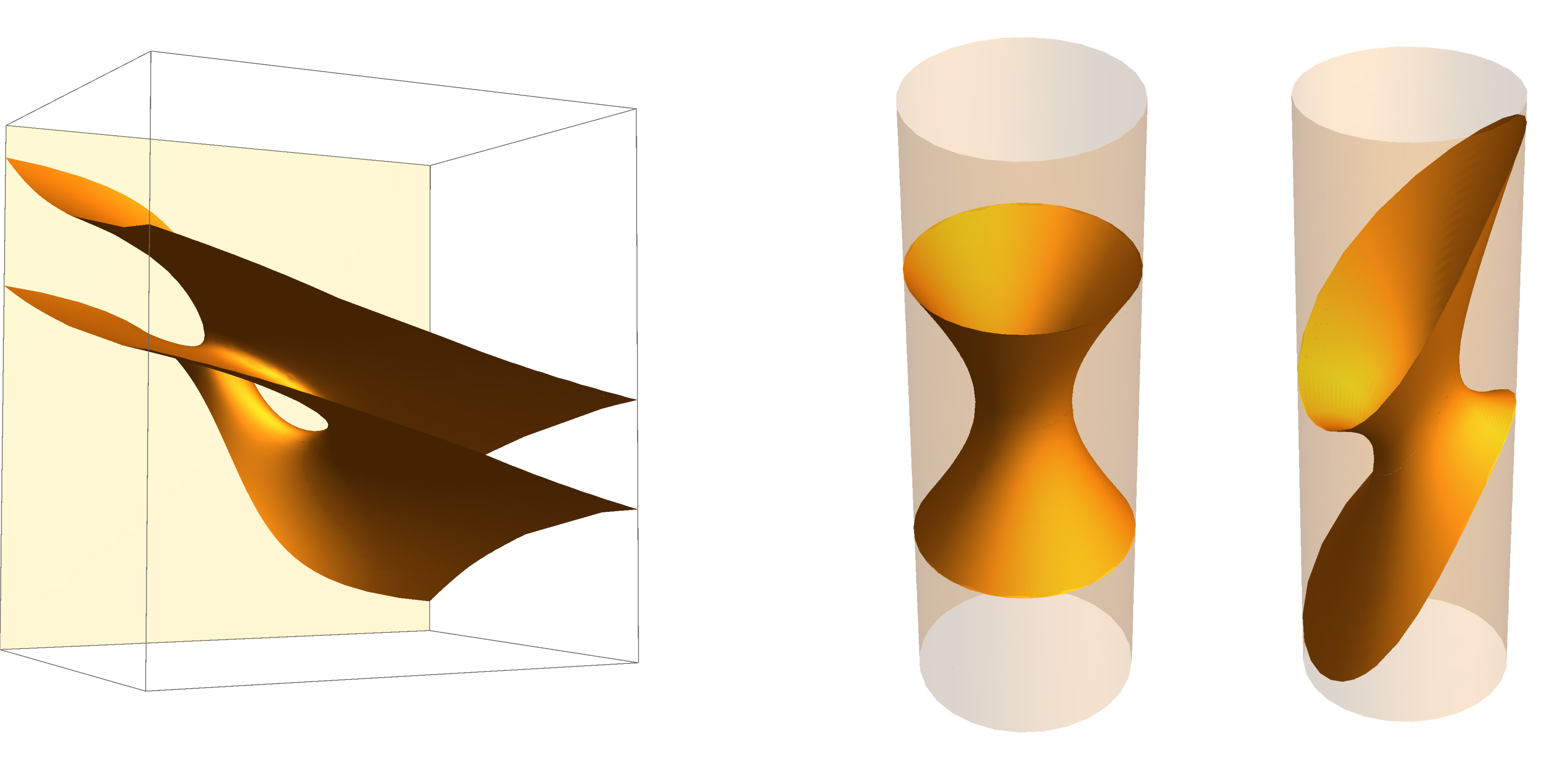} 
 \end{center}
 	\caption{The catenoid for $\tau=1/2$ and $c=10$ in the half space model (left) and the catenoid centered in the origin and a hyperbolically translated copy of it (right).}\label{Catenoids}
 \end{figure}

\end{enumerate}

\section{Minimal annuli }

In this section we will use Douglas criterium in order to prove that there exist minimal annuli with boundary two circles in parallel slices $\{t=0\}$ and $\{t=h\}$ centered at $(0,1,0)$ and $(0,1,h)$ with hyperbolic radius large enough in the half space model of $\PSL$ when $h<\sqrt{1+4\tau^2}\pi$. We also give a non existence result for these annuli when $h>\sqrt{1+4\tau^2}\pi$. We observe that these annuli are not the intersection of rotational catenoid and a slab of height $h$.

\begin{proposition}\label{annulus}
	If $h<\sqrt{1+4\tau^2}\pi$, then there exists a  compact area minimizing (minimal) annulus $A_h$ with boundary two curves contained in parallel horizontal minimal slices  separated by vertical height $h$ in the half space model of $\PSL$. 
\end{proposition}
\begin{proof}
	Let $\gamma_1\subset\{t=0\}$ be  the circle centered at $(0,1,0)$ with hyperbolic radius $\rho$ and let $\gamma_2$ be its translated copy at height $h$. The disk $D\subset\{t=0\}$ bounded by $\gamma_1$ is the unique minimal surface with boundary $\gamma_1$ due to the maximum principle,  and similarly for $\gamma_2$.  Then by Douglas criterium there exists an area minimizing  annulus if there exists an annulus $A$ with boundary $\gamma_1\cup\gamma_2$ such that:
\begin{equation}\label{c: Douglas}
 	2 \text{Area}(D)>\text{Area}(A). 
\end{equation}

The circle of $\h^2$ with hyperbolic radius $\rho$ centered at the point $(0,1)$ in the half space model is given by the equation   ${x^2+y^2+1}=2y\cosh(\rho)$, therefore a parametrization for $D$ is given by $X(x,y)=(x,y,0 )$, where  ${x^2+y^2+1}<2y \cosh(\rho) $.  Denote by $\hat{g}$ the induced metric by $X$. The entries of $\hat g$ in coordinates $(x,y)$ are given by
$\hat{g}_{11}=\frac{1+4\tau^2}{y^2}$, $\hat{g}_{12}=0$ and $\hat{g}_{22}=\frac{1}{y^2}.$ Then:

\[\text{Area}(D)=\displaystyle\int_{} \sqrt{\det \hat{g}}\ dx\ dy =\displaystyle\int_{\cosh(\rho)-\sinh(\rho)}^{\cosh(\rho)+\sinh(\rho)}\displaystyle\int_{-\sqrt{2\cosh(\rho)y -y^2-1}}^{\sqrt{2\cosh(\rho)y -y^2-1}}\dfrac{\sqrt{1+4\tau^2}}{y^2}dx\ dy  \]
\[=2\pi\sqrt{1+4\tau^2}(\cosh(\rho)-1).\]

 Polar coordinates $(r,\theta,t)$ in the cylinder model of $\PSL$ are given by
\[(x(r,\theta),y(r, \theta),t)=( \tanh(r/2)\cos \theta, \tanh(r/2)\sin \theta,t ),\]
which allow us to rewrite the metric  of $\PSL$ as 
\[ ds^2= dr^2+ \sinh^2r d\theta^2 +(-4\tau \sinh^2\left( \frac{r}{2}\right)  d\theta+dt)^2.\]
We consider the annulus  parametrized in polar coordinates  as 
\[Y(r,\theta)=( r, \theta, \pm u(r)+ v(\theta)), \ r\in (\bar\rho,\rho), \ \theta\in[0,2\pi),\]
where 
\[v(\theta)=4\tau \arctan\left(    \frac{\tanh(\rho/2)\sin\theta}{1-\tanh(\rho/2) \cos \theta}          \right)       \]
and   
\[U(r)=\displaystyle\int _{\bar \rho}^{r} \frac{\sinh \bar \rho}{\sqrt{\sinh^2 s-\sinh(\bar \rho)}}ds, \  \   \ u(r)=\sqrt{1+4\tau^2}  U(r). \]

 Observe that $U$ is the height function of a minimal catenoid in $\h^2\times\R$ ($\tau=0$), see for instance Proposition 3.6 in~\cite{Penafiel}, and consequently $\lim_{\bar{\rho}\to \infty}U(\infty)=\frac{\pi}{2}$. Moreover,  the third coordinate of $\psi(Y(\rho,\theta))$ is $u(\rho)$, so  the boundary of the annulus consists of two circles of radius $\rho$ contained in two horizontal minimal slices.

\begin{lemma}
 Assume that $\tau>0$. The function $U$ and $v$ have the following properties:
\begin{enumerate}
	\item $2\pi \int _{\bar \rho}^{\rho} \sqrt{1+U_r^2}\sinh rdr\leq 2\pi \sqrt{\cosh^2\rho-\cosh^2\bar{\rho}}$.
	\item  $-2\tau<v'(\theta)<2\tau e^{\rho}-2\tau$.
	\item If $2e^{-\rho/2}<\theta<2\pi-2e^{-\rho/2}$ then   $-2\tau<v'(\theta)<0$.
\end{enumerate}
\end{lemma}
\begin{proof}
To prove item (1), we estimate
\begin{align*}	
 & 2\pi \displaystyle\int _{\bar \rho}^{\rho} \sqrt{1+U_r^2}\sinh rdr=2\pi \displaystyle\int _{\bar \rho}^{\rho}\frac{\sinh^2(r)}{\sqrt{\sinh^2(r)-\sinh(\bar \rho)^2}}dr
\\
 &\ \ \  =2\pi \cosh \bar \rho\displaystyle\int_1^{\frac{\cosh \rho}{\cosh \bar \rho}} \dfrac{\sqrt{s^2-\frac{1}{{\cosh^2 \bar \rho}}}}{\sqrt{s^2-1}}ds 
< 2\pi \cosh \bar\rho \displaystyle \int_{1}^{\frac{\cosh \rho}{\cosh \bar \rho}}\dfrac{s}{s^2-1}ds
\\
 &\ \ \  =2\pi\sqrt{\cosh^2 \rho-\cosh^2 \bar\rho}.
\end{align*}

As for item (2), we compute $v'(\theta)=\frac{4 \tau  \tanh \left(\frac{\rho }{2}\right) \left(\cos
	(\theta )-\tanh \left(\frac{\rho }{2}\right)\right)}{1-2
	\cos (\theta ) \tanh \left(\frac{\rho }{2}\right)+\tanh
	^2\left(\frac{\rho }{2}\right)}$, which is symmetric with respect to $\theta=\pi$. We have that $v'$ is  decreasing   from $0$ to $\pi$ and  increasing  from $\pi$ to $2\pi$. Therefore:
\[-2\tau <v'(\pi)<v'(\theta)<v'(0)= 2\tau e^{\rho}-2\tau.   \]   

As for item (3), using again the monotonicity we have that \[-2\tau<v'(\theta)<v'(2 e^{-\rho/2})<0.\]	
\end{proof}

We will estimate the area of a half of the annulus $Y$.
The entries of the induced metric $\widetilde{g}$ of $Y$  in coordinates $(r,\theta)$ are given by  $\widetilde g_{11}=1+u'(r)^2$, $g_{22}=\sinh^2r+\left(v'(\theta )-4 \tau  \sinh
^2\left(\frac{r}{2}\right)\right)^2$ and $\widetilde g_{12}=u'(r) \left(v'(\theta )-4 \tau  \sinh
^2\left(\frac{r}{2}\right)\right)$. The area element $W$ satisfies: 
\[W^2=\widetilde g_{11}\widetilde g_{22}-\widetilde g_{12}^2=  \sinh ^2(r) \left(u'(r)^2+1\right)+\left(v'(\theta )+2\tau- 2\tau \cosh(r) \right)^2.  \]

 Assume that $\tau>0$ (this is not restrictive because the area of the annulus in the case of $\mathbb E(-1,-\tau)$ is the same). We have that:
\begin{align*}
\ W^2 &= \sinh ^2(r) \left(u'(r)^2+1\right)+\left(v'(\theta )+2\tau- 2\tau \cosh(r) \right)^2    
\\
&\leq \sinh ^2(r) \left(u'(r)^2+1\right)+  (v'(\theta )+2\tau)^2 +4\tau \cosh^2(r).   
\end{align*}
 If $2e^{-\rho/2}<\theta<2\pi-2e^{-\rho/2}$, then we have the estimate:
\begin{align*}
 W^2&\leq  \sinh ^2(r) \left(u'(r)^2+1\right)+ 4\tau \cosh^2(r)  +4\tau^2              
\\
 &= 4 \tau ^2 \cosh ^2(r)+\sinh ^2(r) \left(\left(1+4 \tau
^2\right) U_r^2+1\right)+ 4\tau^2 
\\
 &= \left(1+4 \tau ^2\right) \left({U_r}^2+1\right) \sinh
^2(r)+4\tau^2+ 4\tau^2.
\end{align*}
Otherwise, the estimate is:
\begin{align*}
 W^2&\leq  \sinh ^2(r) \left(u'(r)^2+1\right)+ 4\tau \cosh^2(r)  +4\tau^2e^{2\rho}              
\\
 &= 4 \tau ^2 \cosh ^2(r)+\sinh ^2(r) \left(\left(1+4 \tau
^2\right) U_r^2+1\right)+ 4\tau^2e^{2\rho} 
\\
&=  \left(1+4 \tau ^2\right) \left({U_r}^2+1\right) \sinh
^2(r)+4\tau^2+ 4\tau^2e^{2\rho}.
\end{align*}

Then the area of a half of the annulus $Y$ can be estimated as:
\begin{align*}
\frac{1}{2}\text{Area}(Y)&= \displaystyle\int_{0}^{2\pi} \displaystyle\int_{\bar{\rho}}^{\rho}W dr d\theta
\\
&\leq\displaystyle\int_{0}^{2e^{-p/2}} \displaystyle\int_{\bar{\rho}}^{\rho}W dr d\theta
+\displaystyle\int_{2e^{-p/2}}^{2\pi-2e^{-p/2}} \displaystyle\int_{\bar{\rho}}^{\rho}W dr d\theta
+\displaystyle\int_{2\pi-2e^{-p/2}}^{2\pi} \displaystyle\int_{\bar{\rho}}^{\rho}W dr d\theta   
\\
 &\leq 2\pi \displaystyle\int_{\bar{\rho}}^{\rho}(\sqrt{1+4\tau^2}\sqrt{1+U_r^2}\sinh(r)+\sqrt{8}\tau) dr + 8\tau(\rho-\bar{\rho}) e^{\rho/2}  
\\
 &\leq 2\pi \sqrt{1+4\tau^2}\sqrt{\cosh^2\rho-\cosh^2\bar{\rho}}+2\sqrt{8}\pi\tau(\rho-\bar\rho)+8\tau (\rho-\bar{\rho})e^{\rho/2}.  
\end{align*}

We want to compare $\frac{1}{2}\text{Area}(Y)$ and $\text {Area}(D)= 2\pi\sqrt{1+4\tau^2}(\cosh \rho -1 )$.
Choosing $\rho=5/4\bar{\rho}$ we have that: 
\begin{align*}
\sqrt{\cosh^2(5/4 \bar{\rho})-\cosh^2\bar{\rho}}&= \cosh(5/4 \bar{\rho})\left( 1-\frac{\cosh^2\bar{\rho}}{\cosh^2(5/4 \bar{\rho})}\right) ^{1/2}    
\\
&=\cosh(5/4 \bar{\rho})\left( 1-1/2\frac{\cosh^2\bar{\rho}}{\cosh^2(5/4 \bar{\rho})} -1/8 \frac{\cosh^4\bar{\rho}}{\cosh^4(5/4 \bar{\rho})} +\dots   \right)       
\\
&=\cosh(5/4 \bar{\rho})-\frac{1}{4}e^{3/4\bar \rho}+o (e^{1/4\bar{\rho}}) 
\end{align*}
 
Therefore, the area can be estimated as:

\[\frac{1}{2}\text{Area}(Y)< 2\pi \sqrt{1+4\tau^2}( \cosh(5/4 \bar{\rho})-\frac{1}{4}e^{3/4\bar \rho}) +o(\bar{\rho} e^{5\bar{\rho}/8}), \]
which is less than $2\pi \sqrt{1+4\tau^2}( \cosh(5/4 \bar{\rho})-1)$ when $\bar\rho $ is large enough. Moreover, the vertical distance between the boundary components of the annulus tends to $ \sqrt{1+4\tau^2}\pi$ as $\bar \rho\to \infty$. Then by Douglas criterium there exists an area minimizing annulus $A$ with boundary $\gamma_1\cup\gamma_2$ for all vertical distances $h\in (h_0, \sqrt{1+4\tau^2}\pi)$. We call $A_h$ the intersection of the annulus $A$ with a slab composed by two minimal slices separated by height $h$.
\end{proof}

The quantity $\sqrt{1+4\tau^2}\pi$ is sharp because, using  the surfaces $u_{d}^\pm$ in Section 3, we can prove that such annuli $A_h$ do not exist for $h\geq  \sqrt{1+4\tau^2}\pi$.

\begin{proposition}
Let $\gamma_1$ and $\gamma_2$ be two closed curves in $\PSL$ in the half space model and assume that  the region  $\{(x,y,t),\ h_0< t< h_0+\sqrt{1+4\tau^2}\pi \}$,  for some $t_0\in\R$, separates $\gamma_1$ and $\gamma_2$. Then there are no  compact minimal surfaces with boundary $\gamma_1\cup \gamma_2$. 	
	
\end{proposition}
\begin{proof}
	
Assume by contradiction that there exists one such surface $M$. Then, consider the family of surfaces $ u_{d}^\pm$ for $d>0$, with asymptotic boundary the two horizontal lines $\{ (x,0,h_0): x\in \R\}$ and $\{ (x,0,h_0+\sqrt{1+4\tau^2}\pi): x\in \R\}$ in $\partial_\infty\PSL$. Note that the surface $u_d^\pm$ cannot intersect the  boundary of $M$. Then for $d$ small enough they do not intersect $M$ either. On the other hand, for $d$ large enough the surface $ u_{d}^\pm$ intersects $M$.  By continuity, there exists $d_0>0$ such that  $u_{d_{0}}^\pm$ and  $M$ are tangent at an interior point of $M$ and $u_{{d_0}}^\pm$ stays locally at one side of $M$, a contradiction to the maximum principle.		
\end{proof}

\section{Asymptotic theorems}
Using the minimal annuli constructed in Section 4 and considering the half space model for $\PSL$, where some hyperbolic translations keep the $t$-coordinate fixed, we can extend to $\PSL$ the ideas of Theorem 2.1 in \cite{ST}.

\begin{proof}[Proof of Theorem \ref{theorem: Asym}]
		
Let $p_0$ be a point in $\Gamma'\cap L$ and assume that $\Pi(p_0)$ is not the point at infinity of $\partial_\infty \h^2$. If there is a vertical segment in $\Gamma'\cap L$, we choose $p_0$ as the middle point of this segment. Up to an isometry, we can assume that $p_0=(0,0,0)$. Then we have that $\Gamma'$ is contained in the region ${\{(x,0,t):\ a< t<b\}}$, with $-\sqrt{1+4\tau^2}\frac{\pi}{2}<a<b< \sqrt{1+4\tau^2}\frac{\pi}{2}$. Consider two points $q_1=(-\epsilon,0,0)$ and $q_2=(\epsilon,0,0)$, and assume that $\Pi(q_1)\in\Pi(\Gamma')$ and $\Pi(q_2)\notin\Pi(\Gamma')$, see Figure \ref{asymptotic}. Let $c$ be the  geodesic of $\h^2$ with ideal points $\Pi(q_1)$ and $\Pi(q_2)$. Let $P=\Pi^{-1}(c)$ be the minimal vertical plane, and let $S_1$ and $S_2$ be the minimal slices $\{t=a\}$ and $\{t=b\}$. Let $G_\epsilon$ be the region of $\mathcal H \backslash \left( P\cup S_1\cup S_2\right) $ that contains $p_0$ in its asymptotic boundary. Let $c_1$ be a geodesic of $\h^2$ joining two interior points of the open arc of $\partial_\infty\h^2$ with endpoints $\Pi(p_0)$ and $\Pi(q_2)$, and let $U$ be the region of $\mathcal H \backslash \left( \Pi^{-1}(c_1)\cup S_2\cup S_3\right) $ between the slices $S_1$ and $S_2$ that does not contain $p_0$ in its asymptotic boundary.

 Assume by contradiction that there exists  a minimal surface $M$ with asymptotic boundary $\Gamma$, and let $M_0=M\cap G_\epsilon$.  For $\epsilon$ small enough we can ensure that:
 	
 \begin{enumerate}
 	\item The asymptotic boundary of $M_0$ is a subarc $\Gamma_0\subset\Gamma'$.
 	\item The arc $\Gamma_0$ is contained in the region ${\{(x,0,t):\ \bar a< t<\bar b\}}$, with $a<\bar a<\bar b<b $, and $M_0$ is strictly contained in the region   $\{(x,0,t):\ \bar a\leq t\leq \bar b \}$.
 	\item The finite boundary of $M_0$ is contained in $P$. Observe that for $\epsilon$ small enough $M_0$ does not intersect the slices $S_1$ and $S_2$.
 	\item The surface $M_0$ does not intersect the region $U$.
 \end{enumerate}
 	\begin{figure}[htb]
 	\begin{center}
 	\vspace{-0.5cm}\includegraphics[height=7cm]{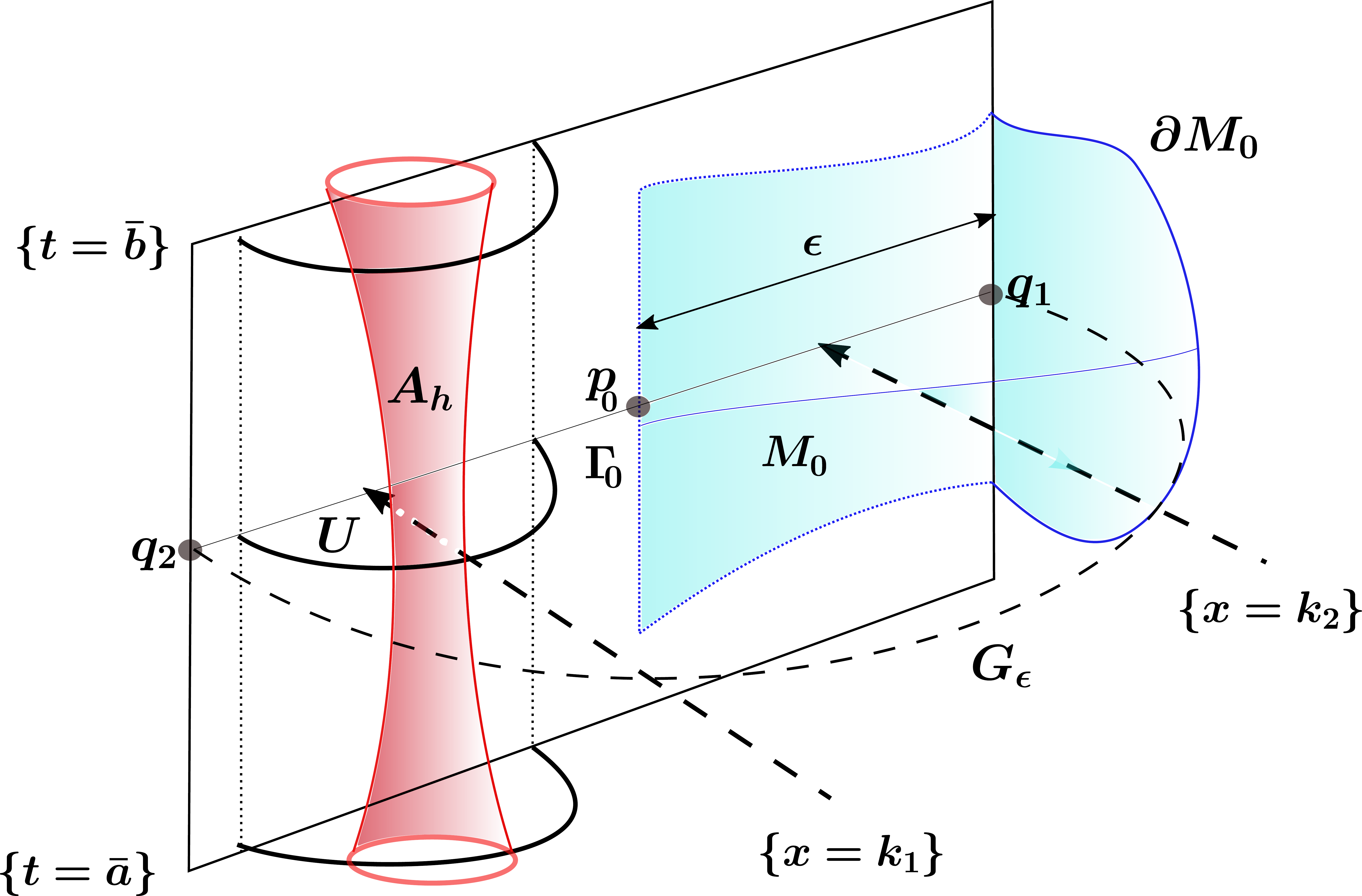}
 	\caption{A plot of the proof. In red the annulus $A_h$ contained in $U$ and in blue the surface $M_0$.} \label{asymptotic}
 	\end{center}
 \end{figure}
 
  Using Proposition \ref{annulus}, we consider a minimal compact annulus $A_h$ with $h=b-a$  with boundary two  curves contained in slices at height $a$ and $b$, respectively.  We can send this annulus to the region $U$ by  hyperbolic translations along the horizontal geodesic   $\{x=k_1\}\subset \h^2$, where $(k_1,0)$ is an interior point in the asymptotic boundary of $\Pi(U)$. Note that these isometries preserve the $t$-coordinate, and the boundary of $A_h$ lies outside the region $\{(x,0,t):\ \bar a< t<\bar b\}$. Now, consider a horizontal geodesic $\{x=k_2\}\subset\h^2$ with $-\epsilon<k_2<0$, and  hyperbolic translations along it, that is,  Euclidean homotheties with center $(k_2,0)$ in this model. We can translate the annulus $A_h$ along this geodesic towards $(k_2,0)$. Observe that the translated copies of the annulus are  contained in the translated copies of $U$, which are in turn contained in $G_\epsilon$. Then, the translated annuli do not intersect the boundary of $M_0$, and also their boundaries  do not intersect the surface $M_0$. Then,  we will achieve a first interior contact point, which contradicts  the maximum principle.
\end{proof}

\begin{remark}\label{R:local}
Note that Theorem \ref{theorem: Asym} is local, so it does not depend on the model. If we consider the cylinder model for $\PSL$ and  $\Gamma$ is a curve in $\partial_\infty\PSL$ as in Theorem \ref{theorem: Asym}, then we can consider the image of this curve by the extended isometry $\psi:\mathcal C\to \mathcal H$ given in Proposition \ref{prop:iso}. We know that if the straight line $L$ is contained in $\Gamma'$ then the same happens for $\psi(\Gamma')$, and if $\Gamma'$ is in one side of $L$ then the same happens for $\psi(\Gamma')$. Considering a subarc $\Gamma''\subset\Gamma'$ we can ensure that $\psi(\Gamma'')$ is contained in $\{(x,0,t),\ t_0<t<t_0+\sqrt{1+4\tau^2}\pi\}$.	
\end{remark}
 We deduce the next corollaries:

 \begin{corollary}\label{cor:thin}
 Let $\Gamma_1:\partial_\infty\h^2\to \partial_\infty\h^2\times\R\subset\partial_\infty\PSL$ be a complete graphical curve parametrized by a complete graph $\Gamma_1(\theta)=(\theta,t(\theta))$. Consider the translated copy $\Gamma_2(\theta)=(\theta,t(\theta)+\sqrt{1+4\tau^2}\pi)$. Then:
 \begin{enumerate}
 	
 	\item There is no  properly immersed minimal surface with asymptotic boundary $\Gamma\subset \partial_\infty\PSL$, being $\Gamma$ a Jordan curve homologous to zero, strictly contained between $\Gamma_1$ and $\Gamma_2$.
 	
 	\item There is no properly immersed minimal surface with asymptotic boundary $\Gamma\subset \partial_\infty\PSL$, being $\Gamma$ a closed curve  strictly contained between $\Gamma_1$ and $\Gamma_2$ whose projection omits an open arc in $\{y=0\}$.
 	
 \end{enumerate}	

  \end{corollary}

 \begin{remark}
 	
This corollary is independent of the model. If we have a curve in this situation in the cylinder model then its image by the isometry $\psi$ given in Proposition~\ref{prop:iso} is in the assumptions of Corollary \ref{cor:thin} in the half space model. Also note that the  curve $\Gamma_2$ can be replaced by another graphical curve such that the vertical height is smaller than $\sqrt{1+4\tau^2}\pi$ pointwise. We state  a particular case when the curves $\Gamma_1$ and $\Gamma_2$ are horizontal circles in the cylinder model:
 \end{remark}

 \begin{corollary}
 There is no properly immersed minimal surface in $\PSL$ with asymptotic boundary  a Jordan curve  $\Gamma\subset \partial_{\infty}\PSL$ homologous to zero, strictly contained between two horizontal circles in $\partial_{\infty}\PSL$ at distance less than $\sqrt{1+4\tau^2}\pi$ in the cylinder model.

\end{corollary}

\subsection{Area minimizing surfaces}

\begin{proof}[Proof of Theorem \ref{minimizing}]

First, note that as in Remark \ref{R:local} this theorem is local so it does not depend on the model. We will assume by contradiction that there exists such surface $M$ in the half space model. 
 
 As $h_\Gamma(p)<\sqrt{1+4\tau^2}\pi$, $p\in I\subset\partial_\infty \h^2$, we can assume up to an ambient isometry that there is a subinterval small enough  $I'=(-\epsilon,\epsilon)\subset I$ such that the asymptotic boundary of $M$ in the region $R=\{(x,0,t): |x|<\epsilon,\  a<t<b\}$ consists of two disjoint curves, where $a,b\in \R$ satisfy $a<b$ and $h_\Gamma (p)<b-a<\sqrt{1+4\tau^2}\pi$ for all $p\in I'$. Let $p_0=(0,0,t_0)$ be a point in the region $R$ strictly contained between the two curves which form the asymptotic boundary of $M$ in the region $R$.

  Let $c$ be the geodesic of $\h^2$ joining the ideal points  $(-\epsilon, 0)$ and $(\epsilon,0)$ of $\partial_\infty \h^2$, let $P=\Pi^{-1}(c)$ be the vertical minimal plane, let $S_1$ and $S_2$ be the slices $\{t=a\}$ and $\{t=b\}$, and let $G_\epsilon$ be the region of $\mathcal H\backslash \left( P\cup S_1\cup S_2\right) $ that contains $p$ in its asymptotic boundary. Let $V$  be a simply connected neighborhood of $p_0$ contained in $G_\epsilon$  such that $M\cap V=\emptyset$.  Let $M_0=M\cap G_\epsilon$, whose asymptotic boundary is contained in the asymptotic boundary of $G_\epsilon$.   As the surface $M$ is proper, choosing $\epsilon$ small enough we can guarantee that there exist $\bar a,\bar b\in \R$ with $a<\bar a<\bar b<b$   such  that:
  
  \begin{enumerate}
  	
  	\item  The asymptotic boundary of $M_0$ in the region $\bar R=\{(x,0,t):\ |x|<\epsilon,\ \bar a<t<\bar b \}$ consists of two disjoint curves which do not intersect $\partial_\infty V$, and $M_0$ is strictly contained in the region   $\{(x,y,t):\ \bar a\leq t\leq \bar b \}$.
  	
  	\item  The finite boundary of $M_0$ is contained in  $P$.
  
  \end{enumerate}    

Consider the area minimizing annulus $A_h$, $h=b-a$,  with asymptotic boundary two curves contained in slices at heights $a$ and $b$ given by Proposition \ref{annulus}. Now, translate  $A_h$  by means of hyperbolic translations along the geodesic  $\{x=0\}$ of $\h^2$, which keep the $t$-coordinate fixed. We can guarantee that there is a translated annulus $\bar A_h$ such that:
\begin{itemize}
 \item $\Pi(\bar A_h)\subset \Pi(V)$,  
 \item the boundary of $\bar A_h$ does not intersect the surface $M_0$, and
 \item $\bar A_h$ intersects the surface $M_0$ in each neighbourhood of the curves which form the asymptotic boundary of $M_0$ in the region $\bar R$, see Figure~\ref{Fig-AreaMinimizing}.
\end{itemize}   In this situation, as $V$ separates $\bar A_h$, there exists at least two compact curves $\gamma_1,\gamma_2\subset \bar A_h\cap M_0$, one above $V$ and the other below $V$. Assume first that these curves are not nulhomotopic in the annulus $\bar A_h$. In this case there is a compact area minimizing surface $\mathcal A\subset \bar A_h$ with boundary $\gamma_1\cup \gamma_2$. We construct a non smooth area minimizing surface by gluing part of the surface $M_0$ with $\mathcal A$ along the curves $\gamma_1$ and $\gamma_2$, which is a contradiction, see Figure~\ref{Fig-AreaMinimizing} left.

 Assume now that $\gamma_1$ is not nulhomotopic in $\bar A_h$, that is, $\gamma_1$ encloses a disk in $\bar A_h$. If $\gamma_1$ also encloses a disk in $M_0$ then replacing one disk by the other we would achieve a contraction. If $\gamma_1$ does not enclose a disk in $M_0$, then there must exist a finite number of  curves $ \gamma_1^i\subset \bar A_h\cap M_0$, $i=1,\dots n$, above $V$ and an area minimizing surface $\mathcal M_0\subset M_0$ with boundary $\gamma_1\cup\gamma_1^1\cup\dots\cup\gamma_1^n$. Then, we repeat the replacement argument, gluing the part of the surface $\bar A_h$ with $\mathcal M_0$ along $\gamma_1, \gamma_1^1,\dots\gamma_1^n$, see Figure~\ref{Fig-AreaMinimizing}, right.       

\end{proof}

\begin{figure}[htb]
	\begin{center}
		\vspace{-0.5cm}\includegraphics[height=7cm]{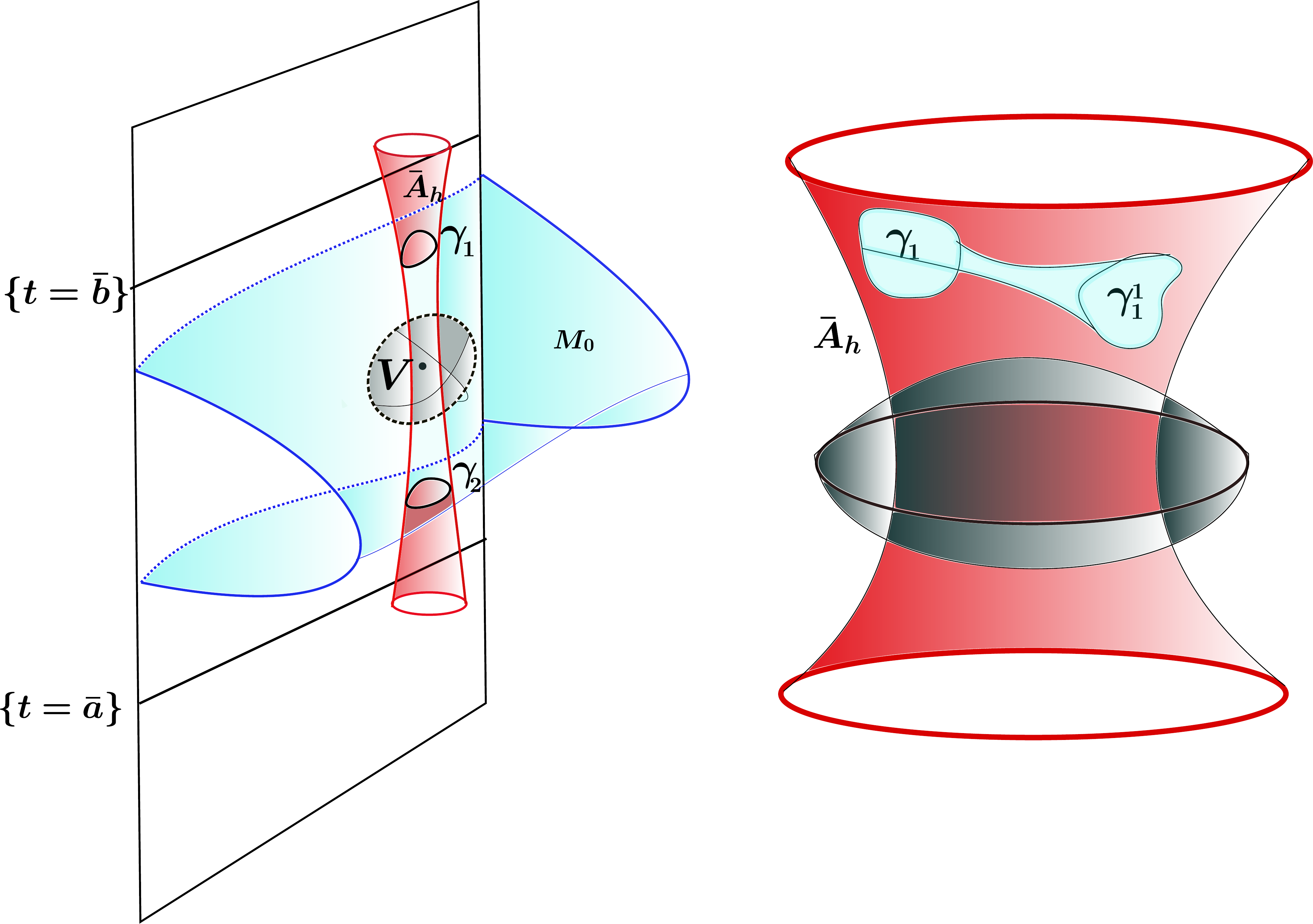}
		\caption{A plot  inside the region $G_\epsilon$, in blue the surface $M_0$, in red the annulus $\bar A_h$ and in gray the neighborhood $V$. At left the case that $\gamma_1$ and $\gamma_2$ are not nulhomotopic in $\bar A_h$. At right the case that $\gamma_1$ is  nulhomotopic $\bar A_h$  } \label{Fig-AreaMinimizing}
	\end{center}
\end{figure}

\section{The Slab Theorem}

This section is devoted to prove Theorem \ref{t:slab}. Recall that  $G_1$ is an entire minimal graph in the cylinder model, whose asymptotic boundary is a closed graphical curve over $\partial_\infty\h^2$ and bounded away from the vertical and we call $G_2$  its translated copy $G_1+(0,0, \sqrt{1+4\tau^2}\pi-\epsilon)$,   being $\epsilon$ a small positive number.
The goal is to construct a continuous family of annuli that play the role of the annuli that gives the map $\Psi$ in Definition \ref{d:slab}, and apply the ideas of \cite{Lima}.

We are going to work in the cylinder model, so we identify $\h^2\equiv \mathbb D$, and $\partial_\infty\h^2\equiv \mathbb S^1$. For two points $\theta_1$, $\theta_2\in\partial_\infty\h^2$ we denote by $(\theta_1,\theta_2)\subset\partial_\infty\h^2$ the arc of $\partial_\infty\h^2$ joining $\theta_1$ with $\theta_2$ in a counterclockwise direction.

\begin{lemma}\label{lemma:an}
	Let $\theta_1$ and $\theta_2$ be two points in $\partial_\infty\h^2$. Let $\gamma\subset\h^2$ be the geodesic joining $\theta_1$ and $\theta_2$, and let $D$ be the connected component of $\h^2\backslash \gamma$ such that $(\theta_1,\theta_2)\subset\partial D$. Then there is a minimal compact annulus $A\subset \Pi^{-1}(D)$ such that the slab region $\{a<t<b\}$, with  $b-a<\sqrt{1+4\tau^2}\pi$, separates the two boundary components of $A$. Moreover, if $p$ is a point in $\Pi^{-1}(D)\cap \{a<t<b\}$ close enough to the asymptotic boundary $\partial_\infty D\times(a, b)$ in Euclidean distance, then we can take  $A$ containing $p$.
\end{lemma}

\begin{proof}
 Consider the minimal compact annulus $A_h$ in the half space model with boundary two curves contained in slices at height $0$ and $h$  with $b-a<h<\pi\sqrt{1+4\tau^2}$ constructed in Proposition \ref{annulus}.  The translated boundary curves contained in the slices can be parametrized as:
		\begin{align*}
		\alpha_\lambda=(\lambda x_1,\lambda y_1,0), (x_1,y_1)\in\gamma_1\subset\h^2, 
		\\ 
		\beta_\lambda=(\lambda x_2,\lambda y_2,h), (x_2,y_2)\in\gamma_2\subset\h^2.
	\end{align*}

	 Then consider the image by the isometry $\phi$ in Proposition \ref{prop:iso}, obtaining:
	\[ \phi\circ \alpha_\lambda=\left( 1-\dfrac{2(1+\lambda y_1 )}{1+2\lambda y_1+\lambda^2(x_1^2+y_1^2)},\dfrac{-2\lambda x_1}{\lambda^2 x_1^2+(1+\lambda y_1)^2},-4\tau \arctan\left( \frac{\lambda x_1}{1+\lambda y_1} \right) \right),   \]
	
	\[ \phi\circ \beta_\lambda=\left( 1-\dfrac{2(1+\lambda y_2 )}{1+2\lambda y+\lambda^2(x_2^2+y_2^2)},\dfrac{-2\lambda x_2}{\lambda^2 x_2^2+(1+\lambda y_2)^2},-4\tau \arctan\left( \frac{\lambda x_2}{1+\lambda y_2} \right)+h \right).  \]
	Hence the vertical gap  between the curves $\alpha_\lambda$ and $\beta_\lambda$ is equal to:
\[
v_\lambda=\max_{(x_1,y_1)\in\gamma_1}\left\lbrace -4\tau \arctan\left( \frac{\lambda x_2}{1+\lambda y_2} \right)+h\right\rbrace-\min_{(x_2,y_2)\in\gamma_2}  \left\lbrace-4\tau \arctan\left( \frac{\lambda x_1}{1+\lambda y_1} \right)\right\rbrace
\]

	As $\gamma_1$ and $\gamma_2$ are compact curves, there exist constants $M_1,\ m_2$ and $M_2$ such that $|x_i|<M_1$ and $0<m_2<y_i<M_2$ for $i=1,2$. Consequently $v_\lambda$ tends to $h$ when $\lambda$ tends to $0$. Given $\epsilon>0$, there exists $\lambda_0>0$ such that for all $\lambda<\lambda_0$, we have that $h-\epsilon<v_\lambda<h+\epsilon$.   
	For all $\epsilon>0$, we can choose  $\lambda>0$ small enough such that $\Pi(A_h)$ is contained in a small neighbourhood of the ideal point $(-1,0)$, and $\phi\circ\alpha_\lambda$ and $\phi\circ\beta_\lambda$ are contained in the regions $\{-\epsilon<t<\epsilon \}$ and $\{h-\epsilon<t<h+\epsilon \}$, respectively. Therefore, we can send the annulus $A_h$ to the region $\Pi^{-1}(D)$  by means of rotations with center the origin, and then  translate the annulus vertically such that the slab region $\{a<t<b\}$ separates the boundary components of $A_h$.	Observe that if the point $p$ is close enough to $\partial_{\infty}D\times (a,b)$, then by choosing the appropriate $\lambda>0$ and the rotation with center the origin we can ensure that $p\in A_h$.    
\end{proof}	

\begin{proof}[Proof of Theorem \ref{t:slab}]	
	
	We identify the asymptotic boundary of $\h^2$ with $\mathbb{S}^1$. The asymptotic boundary of the graphs $G_1$ and $G_2=G_1+(0,0,\sqrt{1+4\tau^2}\pi-\epsilon)$ are curves that can be expressed as the graphs of  continuous functions $g_1,g_2:\mathbb{S}^1\to \R$. Let $M$ be the properly immersed minimal surface contained between $G_1$ and $G_2$ with possibly finite boundary.

	\begin{claim}
		For all $0<\epsilon_1<\frac{\epsilon}{4}$ there exist finitely many   regions $D_i\subset \h^2$ and points $\theta_i\in \partial_\infty D_i$ cyclically ordered, such that:
		\begin{enumerate}
		\item $D_i\cap D_{i+1}\neq \emptyset$,	
		\item $|g_j(\theta_{i})-g_j(\theta)|<\epsilon_1$, for all $\theta\in \partial_\infty D_i$, $j=1,2$,
			\item $M\cap \Pi^{-1}(D_i)\subset \Pi^{-1}(D_i)\cap \{ g_1(\theta_i)-2\epsilon_1<t<g_2(\theta_i)+2\epsilon_1 \}$, 
		\item  $\bigcup_i \partial_\infty D_i=\partial_\infty \h^2$ and
		\item  $\partial M\subset \PSL\backslash\left( \bigcup_i  \Pi^{-1}(D_i)\right) $.
	 	
		\end{enumerate}

	\end{claim}

 Consider $\theta_1$ a point in $\partial_\infty \h^2$.  There exists $\delta_1>0$ small enough such that if   $\gamma_1$ is the geodesic joining the points $\theta_1-\delta_1$, $\theta_1+\delta_1$, and $D_1$ is the  region of $\h^2\backslash \gamma_1$ that contains $\theta_1$, then $|g_j(\theta_1)-g_j(\theta)|<\epsilon_1$ for all $\theta$ in the arc $(\theta_1-\delta_1,\theta_1+\delta_1)$ and $j=1,2$. Choosing $\delta_1$ small enough we have that $M\cap \Pi^{-1}(D_1)\subset \Pi^{-1}(D_1)\cap \{ g_1(\theta_1)-2\epsilon_1<t<g_2(\theta_1)+2\epsilon_1\}$ and  $\partial M\subset \PSL\backslash \Pi^{-1}(D_1)$.

 After that, we choose $\theta_2=\theta_1+\frac{\delta_1}{2}$, and find another $\delta_2$
 such that if   $\gamma_2$ is the geodesic joining the points $\theta_2-\delta_2$, $\theta_2+\delta_2$, and $D_2$ is the  region of $\h^2\backslash \gamma_2$ that contains $\theta_2$, then $|g_j(\theta_2)-g_j(\theta)|<\epsilon_1$ for all $\theta$ in the arc $(\theta_2-\delta_2,\theta_2+\delta_2)$, 
${M\cap \Pi^{-1}(D_2)}\subset \Pi^{-1}(D_2)\cap \{ g_1(\theta_2)-2\epsilon_1<t<g_2(\theta_2)+2\epsilon_1\}$, and $\partial M\subset \PSL\backslash \Pi^{-1}(D_2)$.

  Continuing this process, we construct the sequences  $D_i$ and $\theta_i$. Note that, since the functions $g_j$ are uniformly continuous  we can choose all the $\delta_i>\delta_0$, for some $\delta_0>0$.  As $\partial_\infty \h^2 \equiv \mathbb S^1$ is compact we can ensure that there exist  finitely many $D_1,\ldots,D_n$ and $\theta_1,\ldots,\theta_n$ such that $\bigcup_i \partial_\infty D_i=\partial_\infty \h^2$ and the claim is proved.  
  
   Consider a disk $\mathcal D$  of $\h^2$ centered at the origin  with radius large  enough such that $\h^2\backslash\mathcal D\subset \bigcup_i D_i$. Using Lemma \ref{lemma:an}, we can ensure that there exists an annulus $A_1$ that is contained in the region $\Pi^{-1}(D_1\backslash \mathcal D)$ and whose boundary curves are one above $\{t=g_2(\theta_1)+2\epsilon_1\}$ and the other below $\{t=g_1(\theta_1)-2\epsilon_1\}$.  We can rotate $A_1$  with respect to the origin, keeping it inside $\Pi^{-1}(D_1\backslash\mathcal D)$, until we arrive to the region $\Pi^{-1}((D_1\cap D_2)\backslash \mathcal D)$. Observe that the boundaries of the rotated annuli do not intersect the surface $M\cap \Pi^{-1}(D_1)$. Now apply the vertical translation \[(x,y,t)\mapsto(x,y,t+(g_j(\theta_2)-g_j(\theta_1)))\]  
to this annulus. One of the boundary components of the translated annulus is  above $\{t=g_2(\theta_2)+2\epsilon_1\}$ and the other is below $\{t=g_1(\theta_2)-2\epsilon_1\}$. Observe also that the boundaries of the translated annuli do not intersect the surface since  $|(g_j(\theta_2)-g_j(\theta_1))|<\epsilon_1$. Moreover (as in Lemma \ref{lemma:an}) if $p\in M\cap \Pi^{-1}(D_2)$ is a point close enough to the asymptotic vertical boundary $\partial_\infty\h^2\times\R$, we can translate the annulus toward a point in $\partial_\infty\h^2\times\R$ and rotate  it with respect to the origin until the annulus contains the point $p$. Again the boundaries of all of  these translated annuli do no intersect $M$ because the gap  between the boundaries of the annulus is controlled.

 We call this translated annulus $A_2$. We can iterate these steps to obtain a family of annuli $A_i$ such that:
\begin{enumerate}
	\item $A_i\subset \Pi^{-1}(D_i\backslash\mathcal D), $
	
	\item $\partial A_i\cap M=\emptyset,$
	\item all the  $A_i$ are isometric and there are  smooth maps 
	\[{H_i:[0,1]\to \{A:\ A \text{ minimal annuli in } \PSL\}},\ i=1,\dots,n,\]
	 such that $H_i(0)=A_i$, $H_i(1)=A_{i+1}$ and  $\partial H_i(t) \cap M=\emptyset$.

\end{enumerate}

To conclude the proof we will adapt the ideas in the proof of Lima's Slab Theorem \cite{Lima}. We will indicate the dissimilarities with Lima's proof using his notation. Let $\Sigma$ be a properly immersed minimal surface contained between $G_1$ and $G_2$ with possibly finite boundary ($M=\Sigma$ in Lima's notation). Assume that $\Sigma$  is simply connected ($\partial\Sigma=\emptyset$) or an annulus ($\partial\Sigma\neq \emptyset$). We can do this because we are interested only in the ends of the surface and $\Sigma$ has finite topology.

 Choose a compact subset $B$ (or a metric ball) in $\PSL$ such that:

\begin{enumerate}
	
	\item $\partial \Sigma\subset B$,
\item  $A_i\subset B$ for all $i\in \{1, ...,\ k\}$,
\item $\Sigma\backslash B\subset \bigcup_i  \Pi^{-1}(D_{i}\backslash \mathcal D)$,
\item $\Sigma\cap B$ has a finite number of connected components.

\end{enumerate}

Let  $p_0$ be a point  in $A_1\cap \Sigma\subset B $. Consider a compact subset $K\supset B$ such that any two points of $\Sigma\cap B$ can be joined by a path in $K$.  
	
We will show that, if $p\in \Sigma$ is far enough from $K$, then the tangent plane $T_p\Sigma$ can not be vertical, hence $\Sigma$ is a miltigraph.

To this end, in Lima's proof there are two steps:
\begin{itemize}
	\item Step 1: If $T_p\Sigma$ is vertical, and $\mathcal M$ is a tall rectangle tangent to $\Sigma$ at $p$, then $\Sigma\cap K\neq \emptyset$.
	\item Step 2: There exists $ c>0$ such that, if $d(p,K)>c$, then $\mathcal M\cap K=\emptyset$.  
\end{itemize} 	
 Step 2 easily applies in our case, so we will deal with Step 1.	

Let $\mathcal M$ be a tall rectangle tangent to $\Sigma$ at $p$. Assume that $\mathcal M\cap K=\emptyset$, as in Lima's proof, $\mathcal M$ separates the region between the two entire graphs in two connected components, $\mathcal M(+)$ and $\mathcal M(-)$, and assume that $K\subset M(+)$. There is a neighborhood $U$  of $p$ in $\Sigma$ such that $U\cap \mathcal M$ consists of an equiangular system of at least $2$ curves through $p$. Let $\sigma_1$ and $\sigma_2$ distinct connected component of $U\backslash \Sigma$, contained in $\mathcal M(+)$.

Applying claim 1 in Lima's proof we have that $\sigma_1$ and $\sigma_2$ are contained in  distinct connected component of $\Sigma\cap \mathcal M (+)$, $\Sigma_1$ and $\Sigma_2$.

Let $\mathcal M_\lambda$ be a small hyperbolic translation of $\mathcal M$ such that $\mathcal M_\lambda$ intersects  $\sigma_i$ at $x_i\in \Sigma_i$, $i=1,2$.
We will show that  $S_{x_i}=\Sigma_i\cap \mathcal M_\lambda$ is non compact, $i=1,2$.
 Assume the contrary. Using Lima's ideas we have that both, $S_{x_1}$ and $S_{x_2}$, cannot be compact (see claim 2 in Lima's proof), so we will assume that $S_{x_1}$ is non compact, $S_{x_2}$ is compact and $S_{x_2}\cup \partial \Sigma$ bounds an immersed annulus in $\mathcal M_\lambda(+)$. In this case, we can find a point $z\in S_{x_1}$ arbitrarily far from $\mathcal M$ and close enough to the asymptotic boundary in Euclidean distance.  We have that  $z\in \Pi^{-1}(D_i\backslash \mathcal D)$ for some $i\in \{1,\dots,n\}$.   We can translate the annulus $A_i$ toward a point in the asymptotic boundary and rotate it  with respect to the origin inside the region $\Pi^{-1}(D_i\backslash \mathcal D)$ until  $z$ is contained in the translated and rotated annulus and the annulus is contained in $\mathcal M(+)$.  The boundaries of all of  the translated and rotated annuli do no intersect $\Sigma$ because the gap  between the boundaries of the annuli is controlled when  we translate toward a point in the asymptotic boundary. We call this translated annulus $A_z\subset \mathcal M(+)$. Then, using the properties of the annuli $A_k$, we find a continuous map $f:[0, l]\to \{A:\ A \text{ minimal annuli in } \PSL\}$ such that $f(0)=A_z$, $f(l)=A_{1}$ and  $\partial f(t) \cap \Sigma=\emptyset$. Using the Dragging Lemma in  \cite{CHR}  we will find a path in $\Sigma\cap \mathcal M(+)$ joining $z$ with a point $q\in \Sigma\cap B$.
 Now, join $q$ to a point of $\partial\Sigma$ by a path in $\Sigma\cap K$ and let $\beta$ be the union of these paths. As $K\subset \mathcal M(+)$, then $\beta\subset\Sigma\cap \mathcal M (+)$, which contradicts that $\Sigma_1$ and $\Sigma_2$ are different connected components.

To conclude the Step 1, as $\Sigma_j$, $j=1,2$, is non compact, as before, we consider the point $z_j\in \Sigma_j$ close enough to the asymptotic boundary such that $z_j$ is contained in some $\Pi^{-1}(D_{i_j}\backslash\mathcal D)$, $i_j\in\{1,\dots,n\}$, the annulus $A_{z_j}$ contains the point $z_j$ and $A_{z_j}\subset\mathcal M(+)$. Again, using the properties of the annuli constructed before, and applying the Dragging Lemma, we  find a path in $\Sigma\cap \mathcal M(+)$ connecting $z_j$ with a point $q_j\in\Sigma \cap B$.  Connecting $q_1$ with $q_2$ by a path in $\Sigma\cap K$ we achieve a contradiction because $z_1$ and $z_2$ belong to different  connected components of $\Sigma\cap \mathcal M_ (+)$.    

Step 1 with Step 2 show that each end of $\Sigma$ is a multigraph. For the embedded case we also apply the same ideas of Lima's Theorem. 
\end{proof}

\begin{remark}
	The proof is also valid when $\tau=0$. We have shown here that this estimate is also valid when the two minimal graphs are not necessarily horizontal minimal planes. 
\end{remark}

\section{ Jenkins-Serrin constructions}

\subsection{Twisted Scherk examples}
In this section we follow the ideas developed in~\cite{RP} and show that a similar construction for properly embedded minimal surfaces in $\h^2\times\R$  with  finite total curvature  works in the ambient space $\PSL$.   Such surfaces also can be seen as  solutions of an asymptotic Plateau problem for a specific curve $\Gamma$ composed of vertical straight lines in $\partial_\infty\h^2\times\R$ and horizontal geodesics in {$\h^2\times\{\pm \infty \}$}, see Figure~\ref{polygon}. Consider the cylinder model for $\PSL$. Let $0$ be the origin and denote by $\overline{pq}$ the geodesic segment joining the points $p,q\in \h^2$. Let $\theta_1,\ldots,\theta_{2n+1}$, $n\geq 1$ be ideal points in $\partial\h^2 \equiv \mathbb{S}^1 $ cyclically ordered. Consider $\Omega$  the polygonal domain with edges  $\overline{0\theta_1}$,  $\overline{\theta_1\theta_2},\ldots,  \overline{\theta_{2n}\theta_{2n+1}}$ and  $\overline{\theta_{2n+1}0}$. Consider  the Dirichlet problem for the minimal graph equation  with asymptotic boundary values $+\infty$ over $A_1=\overline{0\theta_1}$ and over $A_{i+1}=\overline{\theta_{2i}\theta_{2i+1}}$ for $i=1,\ldots n$ and $-\infty$ over $B_i=\overline{\theta_{2i-1}\theta_{2i}}$ for $i=1,\ldots,n$ and $B_{n+1}=\overline{\theta_{2n+1}0}$. This problem is known as a Jenkins-Serrin problem, and has solution if and only if there exist horocycles $H_i$ for each ideal point $\theta_i$, $i=1,\ldots, 2{n+1}$ such that $\alpha(\Omega)=\beta(\Omega)$, $2\alpha (\mathcal P)<\gamma(\mathcal P)$ and $2\beta(\mathcal P)<\gamma(\mathcal P)$, for all polygonal domain $\mathcal P$ inscribed in $\Omega$, where $\alpha (\mathcal P)=\sum_i|A_i \cap \mathcal P|$, $\beta (\mathcal P)=\sum_i|B_i \cap \mathcal P|$,  $\gamma (\mathcal P)=|\mathcal P|$ and $|\cdot|$ denotes the hyperbolic length  outside the  horocycles $H_i$,  see for instance \cite{MRR,Me,Younes}. Assume that the ideal points are distributed so that the Jenkins-Serrin problem has a solution, and let $\Sigma$ be the corresponding graph.

  Then, after considering the rotation by $\pi$ over the straight line $\{0\}\times\R$, we can extend $\Sigma$ to a complete minimal surface. This surface has asymptotic boundary an admissible polygon at infinity, i.e., a closed curve composed by $2n+1$ geodesics at $\mathbb{H}\times \{+\infty\}$, $2n+1$ geodesics at  $\mathbb{H}\times \{-\infty\}$  and $4n+2$ vertical straight lines joining the ideal points of the geodesics in $\partial_\infty \h^2\times\R$, see Figure~\ref{polygon}. By Theorem 8 in \cite{HMR}, these surfaces have finite total curvature. Moreover, if the angle that   $\overline{0\theta_1}$ and $\overline{\theta_{2n+1}0}$ make at $0$ is less than or equal to $\pi$, then the surfaces are embedded since they are the union of vertical minimal graphs that do no intersect each other except at  their common boundary $\{0\}\times\R$.

\begin{figure}[htb]
	
	\includegraphics[height=5cm]{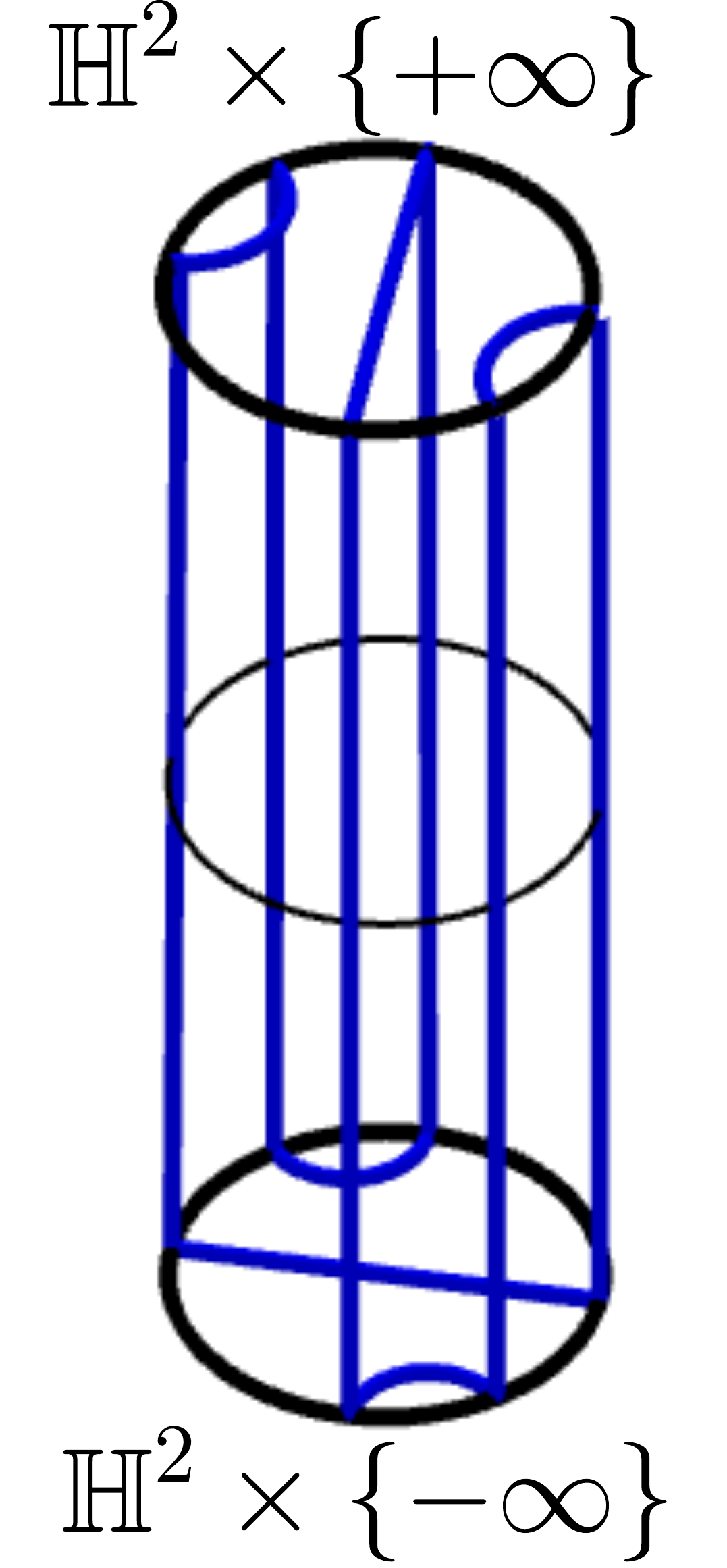} 
	\caption{The admissible polygon at infinity for $n=1$.  }\label{polygon}
\end{figure}

%%%%%%%%%%%%%%%%%%%%%%%%%%%%%%%%%%%%%%%%

\subsection{Helicoidal Scherk examples}
In this section we show that the same construction in \cite{RT} for minimal surfaces in $\h^2\times \R$ can be adapted to $\PSL$. 
Consider the cylinder model for the space $\PSL$.

Let $\theta_1=1$, $\theta_2=e^\frac{i \pi}{2n}$, $n\geq 1$ and $0$ be the origin. Let $\Omega$ be the region bounded by the triangle with edges  the geodesic arcs $\overline{0\theta_1}$, $\overline{\theta_1\theta_2}$ and $\overline{0\theta_2}$. Consider the unique solution to the Jenkins-Serrin problem with boundary values $0$ over $\overline{0\theta_1}$, $+\infty$ over $\overline{\theta_1\theta_2}$ and $h$ over $\overline{0\theta_2}$ for some $h>0$, see for instance \cite{MRR,Me}. As the reflections over horizontal geodesics are isometries, by Schwartz reflection principle  we can consider successive symmetries over the horizontal geodesics, obtaining a simply connected surface $\widetilde{M}_{nh}$ with boundary the vertical straight line $\{0 \}\times \R$. Then after considering the rotation of angle $\pi$ over the straight line $\{0\}\times\R$ we can extend $\widetilde{M}_{nh}$ to a complete minimal surface $M_{nh}$ invariant by the vertical translation $(0,4nh)$ and the screw motion  obtained by composing the rotation of angle $\frac{\pi}{n}$ around $0$ with  the vertical translation $(0,2h)$.
\begin{proposition}
Given $n\geq 1$ and $h>0$,	the surface $M_{nh}$ is a complete embedded minimal surface in $\PSL$  and it is non proper.
\end{proposition}

\begin{proof}
	Let $\Omega_i$ be the domain obtained by rotating $\Omega$ around the origin by an angle $\frac{\pi}{2n}(i-1)$, and let $\widetilde{\Omega}_i=\overline{\Omega_i}\times\R$, $i=1,\dots, 4n$. We will prove that $M_{nh}\cap\widetilde{\Omega}_{1} $ has not self-intersections.
	We have that $\widetilde{M}_{nh}\cap \widetilde{\Omega}_1$ consists of a union of graphs with boundary values:
\[	
	\left\lbrace\begin{array}{l}
	4 k n h\ \text{, over } \overline{0\theta_1}, \\ 
	h+4knh\ \text{, over } \overline{0\theta_2},\\
	+\infty\ \text{, over } \overline{\theta_1\theta_2},\\
	\end{array}\right.
	 	\]
	with $k\in\Z$.  Then, we consider the other union of graphs obtained by  the rotation of angle $\pi$ over the straight line $\{0\}\times\R$, with boundary values:
\[	
	\left\lbrace\begin{array}{l}
	2nh+4knh\ \text{, over } \overline{0\theta_1}, \\ 
	(2n+1)h+4knh\ \text{, over } \overline{0\theta_2},\\
	+\infty\ \text{, over } \overline{\theta_1\theta_2},\\
	\end{array}\right.
	 	\]
	with $k\in\Z$. Hence, $M_{nh}\cap\widetilde{\Omega}_1$ consists of the fundamental piece and its vertical translation by the vector $k(0,2nh)$, $k\in\mathbb{Z}$. We get that $M_{nh}\cap\widetilde{\Omega}_1$ has no self-intersections, and repeating the argument we have the same for $M_{nh}\cap\widetilde{\Omega}_i$. We conclude that the surface $M_{nh}$ is embedded. Note that $M_{nh}\cap\widetilde{\Omega}_1$ accumulates in the vertical plane $\overline{\theta_1\theta_2}\times\R$, and then the surface $M_{nh}$ is not proper.	
\end{proof}

We can also generalize this construction in the following way: Let  $\theta_1=1$ and $\theta_2=e^\frac{i \pi}{2n}$, let $c$ be the shortest arc in $\partial_\infty\mathbb{H}$ joining $\theta_1$ and $\theta_2$, and let  $q_i,\ ... q_m$ be points in the arc $c$ cyclically ordered. Let $\Omega$ be the region bounded by the geodesics arcs $\overline{0\theta_1}$, $\overline{\theta_1q_1}$,...,$\overline{q_iq_{i+q}}$,...,$\overline{q_m\theta_2}$, $\overline{\theta_20}$. Assume that  $\Omega$ satisfies the Jenkins-Serrin conditions with boundary values 0 over $\overline{0\theta_1}$, $h$ over $\overline{\theta_20}$ and alternating $\pm\infty$ over the rest of the edges. Then after considering successive symmetries over the horizontal geodesics and the vertical straight line $\{0\}\times \R$ we obtain a simply connected complete minimal surface embedded which is non proper.
\newline


\begin{thebibliography}{OO}




\bibitem{CasMan} J.\ Castro-Infantes, J.\ M.\ Manzano.  
\newblock Genus one minimal $k-$noids and saddle tower in $\h^2\times\R$.
\newblock Preprint available at	arXiv:2001.07028 [math.DG].



\bibitem{CM} T.\ H.\ Colding, W.\ P.\ Minicozzi II.
\newblock The Calabi-Yau conjectures for embedded surfaces.
\newblock \emph{ Ann.\ of Math.\,} 167 (2008), no. 1, 211--243.



	\bibitem{CHR} P.\ Collin, L.\ Hauswirth, H.\ Rosenberg. 
	\newblock Properly immersed minimal surfaces in a slab of $\mathbb{H}\times\R$, $\mathbb{H}$ the hyperbolic plane.
	\newblock \emph{Arch. Math.\,}
	104 (2015), no. 5, 471--484.
	
	
	
		\bibitem{CHN} P.\ Collin, L.\ Hauswirth, M.\ Nguyen. 
	\newblock Construction of minimal annuli in  ${\widetilde{\mathrm{PSL}}_2(\mathbb{R},\tau)}$ via a variational method.
	\newblock \emph{Preprint}.
	
	
	
	\bibitem{C} B.\ Coskunuzer.
	\newblock Minimal surfaces with arbitrary topology in $\h^2\times\R$.
	\newblock Preprint available at arXiv:1404.0214v2 [math.DG].
	

	


	
 \bibitem{Dan} B.\ Daniel.
 \newblock Isometric immersions into 3-dimensional homogeneous manifolds.
 \newblock \emph{Comment.\ Math.\ Helv.\,} 82 (2007), no. 1, 87--131.	
	
	
	\bibitem{FMMR} L.\ Ferrer, F.\ Martín, R.\ Mazzeo, M.\ Rodríguez. 
	\newblock  Properly embedded minimal annuli in $\h^2\times\R$.
	\newblock \emph{Math.\ Ann.\,} 
	 375 (2019), no. 1-2, 541--594.
		
		
		
		
	
	\bibitem{FP} A.\ Folha, C.\ Peñafiel. 
	\newblock Minimal graphs in ${\widetilde{\mathrm{PSL}}_2(\mathbb{R},\tau)}$.
	\newblock \emph{Mat.\  Contemp.\,}
	43 (2014), 111--132.
			
	\bibitem{HMR} L.\ Hauswirth, A.\ Menezes, M.\ Rodríguez.
	\newblock On the characterization of minimal surfaces with finite total curvature in $\mathbb{H}^2\times\mathbb{R}$ and $\widetilde{PSL}_2(\mathbb{R})$.
	\newblock \emph{Calc. Var. Partial Differential Equations,}
	58 (2019), no. 2, Art 80, 24 pp.
			
			
			
	\bibitem{KM} B.\ Kloeckner, R.\ Mazzeo.
	\newblock On the asymptotic behavior of minimal surfaces in $\h^2\times\R$.
	\newblock 	\emph{Indiana Univ.\ Math.\ J.\,}
	66 (2017), no. 2, 631--658.
				
				
	\bibitem{KMR} P.\ Klaser, A.\ Menezes, A.\ Ramos.
	\newblock On the asymptotic Plateau Problem for area minimizing surfaces in $\mathbb{E}(-1,\tau)$. 
	\newblock  \emph{Ann.\  Global Anal.\  Geom.\,} 58 (2020), no. 1, 1--17.
				
	
	
		 \bibitem{Man} J.\ M. Manzano.
	 \newblock On the classification of Killing submersions and their isometries.
	 \newblock \emph{Pac.\ J. Math.\,} 270 (2014), no. 2, 367--692.
	
	\bibitem{MMR} F.\ Martín, R.\ Mazzeo, M. Rodríguez.
	\newblock Minimal surfaces with positive genus and finite total curvature in $\mathbb{H}^2\times\mathbb{R}$.
	\newblock \emph{Geom.\ Top.\,} 18 (2014), 141--177.
				
	\bibitem{MRR} L.\ Mazet, M. Rodr\'{i}guez, H.\ Rosenberg.
	\newblock The Dirichlet problem for the minimal surface equation --with possible infinite boundary data-- over domains in a Riemannian surface.
	\newblock \emph{Proc.\  London Math.\ Soc.\,} (3) 102 (2011), no. 6, 985--1023.
	
	\bibitem{MPR} W.\ H.\ Meeks III, J.\ Pérez, A.\ Ros.
	\newblock The embedded Calabi-Yau Conjecture for finite genus.
	\newblock Preprint available at	arXiv:1806.03104 [math.DG].
	
					
	\bibitem{Me} S.\ Melo.
	\newblock Minimal graphs in $\widetilde{\mathrm{PSL}}_2(\mathbb{R})$ over unbounded domains.
	\newblock \emph{Bull.\ Braz.\ Math.\ Soc.\,} 45 (2014), no. 1, 91--116.
			
			
	\bibitem{MorRod} F. Morabito, M. Rodr\'{i}guez.
	\newblock Saddle Towers and minimal $k$-noids in $\mathbb{H}^2\times\mathbb{R}$.
	\newblock \emph{J. Inst. Math. Jussieu}, 11 (2012), no. 2, 333--349.					
						
	\bibitem{NR} B.\ Nelli, H.\ Rosenberg.
	\newblock Minimal surfaces in $\mathbb{H}^2\times\mathbb{R}$.
	\newblock\emph{Bull.\ Braz.\ Math.\ Soc.\,} 33 (2002), no. 2, 263-292.
							
							
	\bibitem{Penafiel}C. Peñafiel.
	\newblock Invariant surfaces in ${\widetilde{\mathrm{PSL}}_2(\mathbb{R},\tau)}$ and applications.
	\newblock \emph{Bull.\ Braz.\ Math.\ Soc.(N.S.),}
	43 (2012), no. 4, 545--578.
				
				
				
							
	\bibitem{RP}  J. Pyo, M. Rodríguez.
	\newblock Simply Connected Minimal Surfaces with Finite Total Curvature in $\h^2\times\R$.
	\newblock \emph{Int.\ Math.\ Res.\ Not.\,} 
	IMRN (2014), no. 11, 2944--2954.
								
								
								
	\bibitem{RT} M. Rodríguez, G. Tinaglia.
	\newblock  Non-proper complete minimal surfaces embedded in $\h^2\times\R$.
	\newblock \emph{Int.\ Math.\ Res.\ Not.\,} 
	IMRN (2015), no. 12, 4322--4334.
									
									
								
	
   	\bibitem{ST} R.\ Sa Earp, E.\ Toubiana.
	\newblock  An asymptotic theorem for minimal surfaces and existence results for minimal graphs in $\h^2\times\R$.
	\newblock \emph{Math.\ Ann.\,} 
	342 (2008), no. 2, 309--331.
									
   	\bibitem{Lima} V. Lima.
	\newblock  The slab theorem for minimal surfaces in $\mathbb{E}(-1,\tau)$.
	\newblock \emph{Ann.\ Global Anal.\ Geom.\,} 
	 51 (2017), no. 2, 189--208.
									
									
	\bibitem{Younes} R.\ Younes.
	\newblock Minimal surfaces in  ${\widetilde{\mathrm{PSL}}_2(\mathbb{R},\tau)}.$
		\newblock \emph{Illinois J.\ Math.\,}
		54 (2010), no. 2, 671--712. 						
									
									
									% \bibitem{AMO} L. J. Al\'{i}as, M. A. Mero\~{n}o, I. Ortiz.
									% \newblock On the First Stability Eigenvalue of Constant Mean Curvature Surfaces Into Homogeneous 3-Manifolds.
									% \newblock \emph{Mediterr. J. Math.}, \textbf{12} (2015), no. 1, 147--158.
									
									% \bibitem{Cap} L. Capogna, D. Danielli, S. D. Pauls, J. Tyson.
									% \newblock An introduction to the Heisenberg group and the sub-Riemannian isoperimetric problem.
									% \newblock \emph{Progress in Mathematics}, \textbf{259} (2007). Birkh\"{a}user. ISBN 978-3-7643-8132-5.
									
									% \bibitem{C} S. Cartier.
									% \newblock  Saddle towers in Heisenberg space.
									% \newblock Preprint available at arXiv:1406.6610 [math.DG] (2014).
									
									% \bibitem{CY} S. Y. Cheng, S. T. Yau.
									% \newblock Differential equations on Riemannian manifolds and their geometric applications.
									% \newblock \emph{Comm. Pure Appl. Math.}, \textbf{28} (1975), no. 3, 333--354.
									
									% \bibitem{CY2} S. Y. Cheng, S. T. Yau.
									% \newblock Maximal space-like hypersurfaces in the Lorentz-Minkowski spaces.
									% \newblock \emph{Ann. of Math. (2)}, \textbf{104} (1976), no. 3, 407--419.
									
									% \bibitem{CK} P. Collin, R. Krust.
									% \newblock Le probl\`{e}me de Dirichlet pour l'\'{e}quation des surfaces minimales sur des domaines non born\'{e}s.
									% \newblock \emph{Bull. Soc. Math. France}, \textbf{119} (1991), no. 4, 443--462.
									
									% \bibitem{CR} P. Collin,  H. Rosenberg.
									% \newblock Construction of harmonic diffeomorphisms and minimal graphs.
									% \newblock \emph{Ann. of Math. (2)}, \textbf{172} (2010), no. 3, 1879--1906.
									
									% \bibitem{Dan} B. Daniel.
									% \newblock Isometric immersions into 3-dimensional homogeneous manifolds.
									% \newblock \emph{Comment. Math. Helv.}, \textbf{82} (2007), no. 1, 87--131.
									
									% \bibitem{Dan2} B. Daniel.
									% \newblock The Gauss map of minimal surfaces in the Heisenberg group.
									% \newblock \emph{Int. Math. Res. Not.}, \textbf{2011} (2011), no. 3, 674--695.
									
									% \bibitem{DanHau} B. Daniel, L. Hauswirth.
									% \newblock Half-space theorem, embedded minimal annuli and minimal graphs in the Heisenberg group. 
									% \newblock\emph{Proc. Lond. Math. Soc.}, \textbf{98} (2009), no. 2, 445--470.
									
									% \bibitem{DVVW} F. Dillen, I. Van der Woewtyne, L. Vestraelen, J. Walrave.
									% \newblock Ruled surfaces of constant mean curvature in $3$-dimensional Minkowski space.
									% \newblock Geometry and topology of submanifolds, VIII, 145--147, World Sci. Publ., River Edge, NJ, 1996.
									
									% \bibitem{ER} J. M. Espinar, H. Rosenberg.
									% \newblock Complete constant mean curvature surfaces in homogeneous spaces.
									% \newblock \emph {Comment. Math. Helv.}, \textbf{86} (2011), no. 3, 659--674.
									
									% \bibitem{FerMir} I. Fern\'{a}ndez, P. Mira.
									% \newblock Holomorphic quadratic differentials and the Bernstein problem in Heisenberg space.
									% \newblock \emph{Trans. Amer. Math. Soc.}, \textbf{361} (2011), no. 11, 5737--5752.
									
									% \bibitem{FM} A. Folha, S. Melo.
									% \newblock The Dirichlet problem for constant mean curvature graphs in $\mathbb{H}\times\mathbb{R}$ over unbounded domains.
									% \newblock \emph{Pac. J. Math.}, \textbf{251} (2011), no. 1, 37--65.
									
									% \bibitem{FMP} C. Figueroa, F. Mercuri, R. H. L. Pedrosa.
									% \newblock Invariant surfaces of the Heisenberg groups.
									% \newblock Ann. Mat. Pura Appl. (4), \textbf{177} (1999), 173--194. 
									
									%\bibitem{HST} L.\ Hauswirth, R.\ Sa Earp, E.\ Toubiana.
									%\newblock Associate and conjugate minimal immersions in $M\times\mathbb{R}$.
									%\newblock\emph{Tohoku Math. J.}, \textbf{60} (2008), 267--286.
									
									%\bibitem{Huber} A. Huber.
									%\newblock On subharmonic functions and differential geometry in the large.
									%\newblock \emph{Comment. Math. Helv.}, \textbf{32} (1957), 13--72. 
									
									% \bibitem{JPP} C. Jang, J. Park, K. Park.
									% \newblock Geodesic spheres and balls of the Heisenberg groups.
									% \newblock\emph{Commun. Korean Math.Soc.}, \textbf{25} (2010), no. 1, 83--96.
									
									% \bibitem{Lee} H. Lee.
									% \newblock Extensions of the duality between minimal surfaces and maximal surfaces.
									% \newblock\emph{Geom. Dedicata}, \textbf{151} (2011), 373--386.
									
									% \bibitem{LeeMan} H. Lee, J. M. Manzano.
									% \newblock Generalized Calabi's correspondence and complete spacelike surfaces.
									% \newblock Preprint available at arXiv:1301.7241[math.DG] (2013).
									
									% \bibitem{LR} C. Leandro, H. Rosenberg.
									% \newblock Removable singularities for sections of Riemannian submersions of prescribed mean curvature.
									% \newblock\emph{Bull. Sci. Math.}, \textbf{133} (2009), no. 4, 445--452.
									
									% \bibitem{Man} J. M. Manzano.
									% \newblock On the classification of Killing submersions and their isometries.
									% \newblock \emph{Pac. J. Math.}, \textbf{270} (2014), no. 2, 367--692.
									
									%\bibitem{MPR} J. M. Manzano. J. P\'{e}rez, M. Rodr\'{i}guez.
									%\newblock Parabolic stable surfaces with constant mean curvature.
									%\newblock\emph{Calc. Var. Partial Differential Equations}, \textbf{42} (2011), no. 1--2, 137--152.
									
									%\bibitem{ManNel} J. M. Manzano. B. Nelli.
									%\newblock Height and area estimates for constant mean curvature graphs in $\E(\kappa,\tau)$-spaces.
									%\newblock Preprint (2015). Available at arXiv:1504.05239.
									
									%\bibitem{ManRod} J. M. Manzano, M. Rodr\'{i}guez.
									%\newblock On complete constant mean curvature vertical multigraphs in $\mathbb{E}(\kappa,\tau)$.
									%\newblock\emph{J. Geom. Anal.}, \textbf{25} (2015), no. 1, 336--346.
									
									% \bibitem{Me} S. Melo.
									% \newblock Minimal graphs in $\widetilde{\mathrm{PSL}}_2(\mathbb{R})$ over unbounded domains.
									% \newblock \emph{Bull. Braz. Math. Soc.}, \textbf{45} (2014), no. 2, 91--116.
									
									%\bibitem{MorRod} F. Morabito, M. Rodr\'{i}guez.
									%\newblock Saddle Towers and minimal $k$-noids in $\mathbb{H}^2\times\mathbb{R}$.
									%\newblock \emph{J. Inst. Math. Jussieu}, \textbf{11} (2012), no. 2, 333--349.
									
									% \bibitem{NST} B. Nelli, R. Sa Earp, E. Toubiana.
									% \newblock Minimal Graphs in  $\Nil$: existence and non-existence results.
									% \newblock\emph{In progress} (2015).
									%
									%\bibitem{Plehnert} J. Plehnert.
									%\newblock Constant mean curvature $k$-noids in homogeneous manifolds.
									%\newblock \emph{Illinois J. Math.} (2015), to appear.
									
									% \bibitem{Pyo} J. Pyo.
									% \newblock New complete embedded minimal surfaces in $\H^2\times\R$.
									% \newblock \emph{Ann. Glob. Anal. Geom.}, \textbf{40} (2011), no. 2, 167--176.
									
									%\bibitem{RST} H. Rosenberg, R. Souam, E. Toubiana.
									%\newblock General curvature estimates for stable $H$-surfaces in 3-manifolds and applications.
									%\newblock \emph{J. Differential Geom.}, \textbf{84} (2010), no. 3, 623--648.
									
									%\bibitem{Treibergs} A. Treibergs.
									%\newblock Entire Spacelike Hypersurfaces on Constant Mean Curvature in Minkowski Space.
									%\newblock \emph{Invent. Math.}, \textbf{66} (1982), no. 1, 39--56.


\end{thebibliography}
\end{document}